\documentclass[twoside]{article}

\usepackage[accepted]{aistats2023}
\usepackage[round]{natbib}

\usepackage[utf8]{inputenc} 
\usepackage[T1]{fontenc}    
\usepackage{hyperref}       
\usepackage{url}            
\usepackage{booktabs}       
\usepackage{amsfonts}       
\usepackage{nicefrac}       
\usepackage{microtype}      
\usepackage{xcolor}         

\usepackage{amsmath}
\usepackage{amssymb}
\usepackage{mathtools}
\usepackage{amsthm}
\usepackage{enumitem}
\usepackage{booktabs}
\usepackage{mathtools}				
\DeclareMathOperator{\rk}{rank}	
\DeclareMathOperator{\rank}{rank}	
\DeclareMathOperator{\tr}{tr}	    
\DeclareMathOperator{\vecc}{vec}	    
\DeclareMathOperator{\mat}{mat}	    
\DeclareMathOperator{\st}{s.t.}
\DeclareMathOperator*{\argmin}{arg\,min}        

\newcommand{\RR}{\mathbb R}
\DeclareMathOperator{\RIP}{RIP}

\DeclarePairedDelimiter{\norm}{\lVert}{\rVert}

\newcommand{\projr}{\mathcal{P}_r}

\DeclareMathOperator{\dist}{dist}
\def\de{{\rm d}}

\usepackage{amsthm}
\newtheorem{theorem}{Theorem}

\newtheorem{corollary}{Corollary}
\newtheorem{lemma}{Lemma}
\theoremstyle{definition}
\newtheorem{definition}{Definition}

\newtheorem{assumption}{Assumption}
\theoremstyle{remark}

\usepackage{subcaption}

\begin{document}

%
\runningtitle{Noisy Low-rank Matrix Optimization}

%

\twocolumn[

\aistatstitle{Noisy Low-rank Matrix Optimization: \\ Geometry of Local Minima and Convergence Rate}

\aistatsauthor{ Ziye Ma \And Somayeh Sojoudi}

\aistatsaddress{Department of EECS \\University of California, Berkeley \And Department of EECS \\
   Department of Mechanical Engineering \\
   University of California, Berkeley } ]

\begin{abstract}
  This paper is concerned with low-rank matrix optimization, which has found a wide range of applications in machine learning. This problem in the special case of matrix sensing has been studied extensively through the notion of Restricted Isometry Property (RIP), leading to a wealth of results on the geometric landscape of the problem and the convergence rate of common algorithms. However, the existing results can handle the problem in the case with a general objective function subject to noisy data only when the RIP constant is close to 0. In this paper, we develop a new mathematical framework to solve the above-mentioned problem with a far less restrictive RIP constant. We prove that as long as the RIP constant of the noiseless objective is less than $1/3$, any spurious local solution of the noisy optimization problem must be close to the ground truth solution. By working through the strict saddle property, we also show that an approximate solution can be found in polynomial time. We characterize the geometry of the spurious local minima of the problem in a local region around the ground truth in the case when the RIP constant is greater than $1/3$. Compared to the existing results in the literature, this paper offers the strongest RIP bound and provides a complete theoretical analysis on the global and local optimization landscapes of general low-rank optimization problems under random corruptions from any finite-variance family.
\end{abstract}

\section{INTRODUCTION}
In this work, we focus on the problem of noisy matrix optimization:
\begin{equation}
	\label{eq:unfactored_noisy_problem}
	\min_{M \in \mathbb{R}^{n \times n}} f(M,w) \ \ \text{s.t.} \ \rk(M) \leq r, M \succeq 0,
\end{equation}
where the objective $f$ takes in two input variables: a low-rank, positive semidefinite matrix $M \in \mathbb{R}^{n \times n}$ and a random variable $w \in \mathbb{R}^m$ that represents some corruption to the objective function. The noise can come from any arbitrary distribution as long as it has a finite variance. We denote the maximum rank of the variable $M$ to be $r$. We optimize \eqref{eq:unfactored_noisy_problem} only with respect to the first variable $M$, while $w$ is assumed to be hidden to the user. The randomness of the parameter $w$ comes from the stage prior to solving \eqref{eq:unfactored_noisy_problem}, which accounts for uncertainty in the model/data or external  factors. Therefore, when the non-convex low-rank optimization is performed, $w$ will not change anymore even though it is unknown to the user. Let $M^*$ be a rank-$r$ matrix that minimizes the function $f(M,0)$ subject to the constraints in \eqref{eq:unfactored_noisy_problem}. This problem has a wide range of applications, the most notable ones being matrix sensing \citep{recht2010guaranteed}, matrix completion \citep{candes2009exact}, and robust PCA \citep{candes2011robust}. This formulation also has extensive applications in recommender systems \citep{koren2009matrix}, motion detection \citep{fattahi2020exact,anderson2019global}, phase synchronization/retrieval \citep{singer2011angular,boumal2016nonconvex,shechtman2015phase}, and power system estimation \citep{zhang2017conic}. The matrix $M^*$ is called the ground truth solution since the objective function is set up to be nonnegative and that $f(M^*,0) = 0$ for most of the above-mentioned applications. The goal is to find the matrix closest to $M^*$ in terms of Frobenius norm under the rank constraint. However, the influence of noise is not well studied in the literature due to the complications it may bring.

The major innovation of this work is the analysis of the effect of noise, where the objective function is subject to random corruption that is unknown to the user. This formulation is important yet oftentimes glossed over due to its challenging mathematical analysis, partly due to the sophisticated relationship between each globally optimal solution $M$ and the vector $w$. For instance, consider the canonical example of the matrix sensing problem, where the objective function is quadratic:
\begin{equation}\label{eq:linear_noisy_problem}
\begin{aligned}
\min_{M \in \RR^{n \times n}} \quad & \frac{1}{2}\norm{\mathcal A(M)- \tilde b}^2 \\
\st \quad & \rank(M) \leq r, \quad M \succeq 0.
\end{aligned}
\end{equation}
Here, $\mathcal A:\RR^{n \times n} \to \RR^m$ is a linear operator whose action on the matrix $M$ is given by
\[
\mathcal{A}(M)=[\langle A_1, M \rangle, \dots, \langle A_m, M \rangle]^\top,
\]
where $A_1,\dots,A_m \in \mathbb{R}^{n \times n}$. $\tilde b=\mathcal A(M^*) + w$ represents perfect measurements on some ground truth $M^*$ plus some noise $w$. The user only observes $\tilde b$ and has no access to noieless measurements, which means that the matrix $M^*$ of interest is the global minimum of \eqref{eq:linear_noisy_problem} only when $w=0$. When $w \neq 0$, the global minimum of \eqref{eq:linear_noisy_problem} would likely differ from $M^*$. In this case, it is desirable to study whether local search algorithms can converge to a point that is close to $M^*$ with high probability. Other applications such as matrix completion and robust PCA also suffer from the same conundrum since they all aim to align a given matrix to some partially observed matrix that is corrupted by unknown noise. In real-life problems, the corruptions induced by noise cannot be ignored or circumvented because they usually come from physical sources. For instance, in the power grid state estimation problem, which can be formulated as matrix sensing \citep{jin2019towards}, measurements come from physical devices and the noise can be originated from mechanical failures, bad weather, and even cyber-attacks.

Due to the existence of a rank constraint, the optimization problem \eqref{eq:unfactored_noisy_problem} is non-convex. Thus, local search algorithms can potentially converge to poor local minimizers, defeating the purpose of solving \eqref{eq:unfactored_noisy_problem}. Although \eqref{eq:unfactored_noisy_problem} may be solved via convex relaxations for different classes of $f(\cdot,\cdot)$ to overcome the non-convexity issue when $f(\cdot,0)$ is quadratic \citep{candes2009exact,recht2010guaranteed,candes2010power}, the computational challenge associated with solving semidefinite programming problems is prohibitive for large-scale problems. This has inspired many papers to solve \eqref{eq:unfactored_noisy_problem} via the Burer-Monteiro factorization \citep{burer2003nonlinear} by factoring $M$ into $XX^\top$, where $X \in \mathbb{R}^{n \times r}$, since $M$ is positive semidefinite and has rank at most $r$. By doing so, one can convert the constrained optimization \eqref{eq:unfactored_noisy_problem} into an unconstrained problem. Specifically, we solve the following problem instead of \eqref{eq:unfactored_noisy_problem}:
\begin{equation}
	\label{eq:main_noisy_problem}
	\min_{X \in \mathbb{R}^{n \times r}} f(XX^\top,w)
\end{equation}
The main issue with \eqref{eq:main_noisy_problem} is that it is still a non-convex problem, despite being more scalable and easier to deal with computationally. To address this issue, a popular line of research in the literature is to study the optimization landscape of \eqref{eq:main_noisy_problem}. Namely, the goal is to find the distance between the furthest local minimum and the global minimum, in addition to studying the convergence rate of local search methods in terms of the geometry of the optimization landscape.

This line of work usually assumes the Restricted Isometry Property (RIP) for the problem, which is defined below:
\begin{definition}
Given a fixed parameter $w$ and integers $r_1$ and $r_2$, the function $f(\cdot,w): \mathbb{R}^{n \times n} \times \mathbb{R}^m \mapsto \mathbb{R}$ is said to satisfy the restricted isometry property of rank $(2r_1,2r_2)$ with the constant $\delta \in [0,1)$, denoted as $\delta$-RIP$_{2r_1,2r_2}$, if the inequality
\[
(1-\delta)\|N\|_F^2 \leq [\nabla^2_M f(M,w)] (N,N) \leq (1+\delta) \|N\|_F^2,
\]
holds for all $M, N \in \mathbb{R}^{n \times n}$ with $\rk(M) \leq 2r_1$ and $\rk(N) \leq 2r_2$.
\end{definition}
Note that $[\nabla^2 f(\cdot,\cdot)]$ is a quadratic form from $\mathbb{R}^{n \times n} \times \mathbb{R}^{n \times n}$ to $\mathbb{R}$. The precise definition will be given in Section \ref{sec:notation}. Usually, the RIP property is defined for sensing operators under the matrix sensing setting, but here we generalize this notion to any arbitrary objective function under noise. More details of RIP may be found in \cite{recht2010guaranteed,zhu2018global}, and it is also known as $(1-\delta)$-restricted strong convexity and $(1+\delta)$-restricted smoothness.

\subsection{Related Works}
We first discuss the line of work that focuses on certifying the in-existence of spurious local minima in the noiseless setting (a local minimum that is not a global minimum is called {\it spurious}). \cite{bhojanapalli2016global} analyzes the absence of spurious local minima under the RIP condition for the matrix sensing problem, or in other words when $f(\cdot,w)$ is quadratic. This study states that $\delta \leq 1/5$ is a sufficient condition. \cite{zhu2018global,li2019non} investigate arbitrary objective functions under the RIP constant $\delta \leq 1/5$. The series of work \citep{zhang2019sharp,zhang2020many} show that the bound $\delta < 1/2$ is a sharp bound for guaranteeing the absence of spurious local minima in the case when the objective function is quadratic. The state-of-the-art result for general objective functions is proved in \cite{bi2021local}'s paper, which states that $\delta < 1/2$ is also sufficient for the absence of spurious local minima.

In the noisy case, \cite{zhang2018primal} proves that all local minima are close to the ground truth when $\delta \leq 1/35$ for a general objective, which is an extremely strong assumption on $\delta$. Furthermore, \cite{zhang2018primal} requires the RIP condition to be satisfied for the noisy problem rather than its noiseless counterpart, which is impossible to verify beforehand due to the unknown noise. For specific objective functions in the form of \eqref{eq:linear_noisy_problem}, \cite{ma2021sharp} shows that $\delta < 1/2$ is sufficient and necessary for the absence of spurious local minima, even when $w$ is sampled from an arbitrary finite-variance family. The major difference between \cite{ma2021sharp} and this work is that we focus on a general objective, while \cite{ma2021sharp} only focuses on a quadratic objective (the matrix sensing objective).

In terms of the convergence of local search methods, \cite{wang2017unified} proves that the gradient descent algorithm applied to \eqref{eq:main_noisy_problem} converges linearly when the initialization is good, given that $\delta < 1/7$. Similarly to \cite{zhang2018primal}, this RIP bound is given with respect to the noisy problem rather than its noiseless version, which is an undesirable feature. For a general noiseless objective (in the case when $w=0$), \cite{bi2021local} proves that there exists a region around $M^*$ in which linear convergence can be established. On the other hand, \cite{bi2021local} also proves that a RIP constant of $\delta < 1/2$ is sufficient for the global establishment of the strict saddle property for a general noiseless objective function. As noted in \cite{JGNK2017}, the strict saddle property can lead to a polynomial convergence to a global optimum with a random initialization. The exact definition of the strict saddle property can be found in \cite{ge2017no}, and it basically states that all approximate local optima must be close to the global optima. In the noisy setting, \cite{ma2021sharp} demonstrates that $\delta < 1/2$ is necessary and sufficient to the establishment for the strict saddle property for quadratic objective functions.

\subsection{Main Contributions}
The contribution of this work is fourfold:
\begin{enumerate}
	\item First, we show that if the noiseless objective function $f(\cdot,0)$ satisfies the $\delta$-RIP$_{2r,2r}$ property with $\delta < 1/3$, then all local minima of the noisy objective \eqref{eq:main_noisy_problem} are in the vicinity of the ground truth solution $M^*$, where the distance to $M^*$ is proven to be a function of the noise intensity and $\delta$. The state-of-the-art result requires that $f(\cdot,w)$ satisfy the $\delta$-RIP$_{6r,6r}$ property with $\delta < 1/35$, which is a much stronger assumption than ours. Moreover, since the RIP constant of $f(\cdot,w)$ is impossible to verify due to the randomness of $w$, we impose the RIP condition on the noiseless function $f(.,0)$.
	\item In the case when $\delta \geq 1/3$, we show that there is a local area around the ground truth solution $M^*$ such that any local minimum of \eqref{eq:main_noisy_problem} in this region must be very close to $M^*$. The size of this local area is parametrized by some constant $\tau$ with the property that as $\tau$ decreases, the local minima will be more tightly concentrated around $M^*$.
	\item We prove that \eqref{eq:main_noisy_problem} exhibits the strict saddle property globally when $\delta < 1/3$. This means that there exists an algorithm that can reach the global minimum in polynomial time with a random initialization.
	\item Finally, it is proved that there exists a region around the global minimum in which the vanilla Gradient Descent algorithm converges linearly on \eqref{eq:main_noisy_problem} for any arbitrary $\delta$. This result was previously established in the literature only when the RIP constant of $f(\cdot,w)$ is smaller than $1/7$.
\end{enumerate}

To highlight our improvements over the existing results, Table~\ref{table:results_comparison} lists some state-of-the-art comparable works to showcase the strength of the guarantees provided in this paper. Note that when we denote the objective function as "General Noisy", it means that the function is in the form of \eqref{eq:main_noisy_problem} and satisfies the RIP property. We further denote the objective function of \eqref{eq:linear_noisy_problem} as "Quadratic Noisy", which is also known as the matrix sensing problem. In particular, according to \cite{candes2011tight}, $\mathcal{O}(1/\delta^2)$ number of random Gaussian measurements are required to ensure $\delta$-RIP$_{2r}$, so a RIP constant of $1/3$ vs $1/35$ introduces a difference in sample number requirement of around $35^2/3^2 \approx 100$ times. Also since \cite{zhang2018primal} requires $\delta$-RIP$_{6r}$, even more measurements are needed.

\begin{table*}
\centering
\renewcommand*{\arraystretch}{1.5}
\caption{Comparison between our result and the prior literature.}
\begin{tabular}{lllll}
\toprule
\multicolumn{1}{l}{Paper} & \multicolumn{1}{l}{Objective function} & \multicolumn{1}{l}{Quality of Local Min} & \multicolumn{1}{l}{Strict Saddle} &\multicolumn{1}{l}{Convergence} \\\midrule
\cite{zhang2018primal} & General Noisy & $\delta < 1/35$ & N/A & N/A                     \\
\cite{wang2017unified} & General Noisy & N/A & N/A & Linear rate with $\delta<1/7$  \\
\cite{ma2021sharp} & Quadratic Noisy & $\delta \leq 1/2$ & $\delta<1/2$ & N/A \\
Ours & General Noisy & $\delta<1/3$ &$\delta<1/3$ & Linear rate with arbitrary $\delta$  \\
\bottomrule
\end{tabular}
\label{table:results_comparison}

\end{table*}

\section{PRELIMINARIES}

\subsection{Assumptions on the objective function}
\label{sec:assump}
The assumptions stated in this section serve as the underpinnings of all the theorems in this paper, and they mainly require that the objective function be smooth with respect to both the decision variable $X$ and the noise $w$. To clarify, these assumptions do not pose any restriction on $w$, and this parameter can come from any probability distribution.
\begin{description}
	\item[Assumption 1.] The objective function $f(\cdot,\cdot)$ is twice continuously differentiable with respect to its first argument $M$.
	\item[Assumption 2.] The noiseless objective function $f(\cdot,0)$ satisfies the $\delta$-RIP$_{2r,2r}$ property for some constant $\delta \in [0,1)$.
	\item[Assumption 3.] The noise $w$ has a finite influence on the gradient and Hessian of the objective function in the sense that there exist two constants $\zeta_1 \geq 0$ and $\zeta_2 \geq 0$ such that
	\begin{align}
\begin{split}
&| \langle \nabla_M f(M,w) - \nabla_M f(M,0), K \rangle | \leq \\&\zeta_1 \|w\|_2 \|K\|_F,
\end{split}\label{eq:noise_perturb_grad} \\
\begin{split}
&| [\nabla^2_M f(M,w) - \nabla^2_M f(M,0)](K,L) | \leq  \\
&\zeta_2 \|w\|_2 \|K\|_F \|L\|_F
\end{split}\label{eq:noise_perturb_hess}
\end{align}
for all matrices $M,K,L \in \mathbb{R}^{n \times n}$ with $\rk(M), \rk(K), \rk(L) \leq 2r$.

\end{description}

As an example, for the standard matrix sensing problem \eqref{eq:linear_noisy_problem} with the sensing matrix $\mathcal{A}$, if $\mathcal{A}$ satisfies the RIP property, then all of these assumptions hold with $\zeta_1 = \|\mathcal{A}\|_2$ and $\zeta_2 =0$. The 1-bit matrix completion problem is also an example that satisfies the above assumptions which will be elaborated in Section \ref{sec:num_conv}. Note that although in our problem statements we assume $M$ to be symmetric and positive semidefinite, our framework can also be adapted to deal with non-symmetric and non-square matrix $M$. A more detailed discussion is provided in Appendix \ref{sec:appendix_asymmetric}.

\subsection{Notation}
\label{sec:notation}
In this paper, $I_n$ refers to the identity matrix of size $n \times n$. The notation $M \succeq 0$ means that $M$ is a symmetric and positive semidefinite matrix. $\sigma_i(M)$ denotes the $i$-th largest singular value of a matrix $M$, and $\lambda_i(M)$ denotes the $i$-th largest eigenvalue of $M$. $\norm{v}$ denotes the Euclidean norm of a vector $v$, while $\norm{M}_F$ and $\norm{M}_2$ denote the Frobenius norm and the operator norm, respectively. The inner product $\langle A,B \rangle$ is defined to be $\tr(A^\top B)$ for two matrices $A$ and $B$ of identical dimensions. For a matrix $M$, $\vecc(M)$ is the usual vectorization operation by stacking the columns of the matrix $M$ into a vector. The Hessian of the function $f(\cdot,\cdot)$ in \eqref{eq:main_noisy_problem} with respect to the first argument $M$, denoted as $\nabla^2_M f(\cdot,\cdot)$, can be regarded as a quadratic form whose action on any two matrices $K,L \in \RR^{n \times n}$ is given by
\[
[\nabla^2_M f(M,w)](K,L)=\sum_{i,j,k,l=1}^n\frac{\partial^2f}{\partial M_{ij}\partial M_{kl}}(M,w)K_{ij}L_{kl}.
\]
In this paper, $\nabla^2_M f(M,w)$ and $\nabla^2 f(M,w)$ are used interchangeably since $w$ is an unknown fixed parameter and it is impossible to take a derivative with respect to $w$.

Define $M^* \in \argmin_M f(M,0)$. We also characterize the distance of an arbitrary factorized point $X \in \mathbb{R}^{n \times r}$ to a rank-$r$ positive semidefinite matrix $M$ with the function $\dist(X,M)$, defined as:
\begin{align*}
	&\dist(X,M) = \min_{Z \in \mathcal{Z}} \|X-Z\|_F, \\
	 &\mathcal{Z} = \{Z \in \mathbb{R}^{n \times r}\ | \  M = ZZ^\top \}.
\end{align*}
Given a matrix $\hat X \in \mathbb{R}^{n \times r}$, define $\hat{\mathbf X} \in \mathbb R^{n^2 \times nr}$ to be the matrix satisfying
\[
	\hat{\mathbf X} \vecc(U) = \vecc(\hat XU^\top +U\hat X^\top), \quad \forall U \in \mathbb{R}^{n \times r}.
\]
Define $\projr(M)$ of an arbitrary matrix $M$ to be the projection of $M$ on a low-rank manifold of rank at most $r$:
\begin{align*}
	&\projr(M) = \argmin_{M_r \in \mathcal{M}} \| M_r -M\|_F, \\
	&\mathcal{M} \coloneqq \{M \in \mathbb{S}^{n \times n} \vert \rk(M) \leq r, M \succeq 0 \}
\end{align*}
For problem \eqref{eq:linear_noisy_problem}, $\mathbf A \in \mathbb{R}^{m \times n^2}$ is defined such that $\mathbf A \vecc(M) = \mathcal{A}(M)$.

Finally, define:
\[
	h(X,w) \coloneqq f(X X^\top,w).
\]

\section{GEOMETRY OF LOCAL MINIMA}

\subsection{When $\delta < 1/3$}
\label{sec:location_global}
When the RIP constant $\delta$ is smaller than $1/3$, we show that all local minima (or second-order critical points) of \eqref{eq:main_noisy_problem} are close to the ground truth solution $M^*$. The proximity to the ground truth is parametrized by the noise intensity defined as $q \coloneqq \|w\|_2$. When $q=0$, our result (Theorem \ref{thm:global_local_min}) recovers the results previously proved in \cite{ha2020equivalence,zhang2021general}.

\begin{theorem}
	\label{thm:global_local_min}
	Assume that the objective function of \eqref{eq:main_noisy_problem} satisfies Assumptions 1-3 and that $f(M,0)$ satisfies the RIP property with some $\delta$-RIP$_{2r,2r}$ constant such that $\delta < 1/3$. For every $ \epsilon \in [0,\frac{1/3-\delta}{\zeta_2})$, with probability at least $\mathbb P(\norm{w}_2
 \leq \epsilon)$, every local minimizer $\hat X$ of \eqref{eq:main_noisy_problem} satisfies:
	 \begin{equation}
	 	\label{eq:global_local_min_range}
	 	\|\hat X \hat X^\top - M^*\|_F \leq \frac{2 \zeta_1 \epsilon}{1-3(\delta+\zeta_2 \epsilon)}.
	 \end{equation}
\end{theorem}
This is a powerful theorem stating that as long as $\delta < 1/3$, all local minima are close to the ground truth solution, regardless of the family from which $w$ is sampled. Previously, the problem needed to satisfy $\delta < 1/35$ for similar results to hold. Furthermore, unlike \cite{bi2020global}, we achieved this result without the BDP assumption or requiring $r=1$. The upper bound in \eqref{eq:global_local_min_range} is a function of $\epsilon$ and $\delta$. The bound becomes loser as $\epsilon$ and $\delta$ increase. Note that $\zeta_1$ and $\zeta_2$ affect $\epsilon$ in a linear way and therefore obtaining non-conservative constants $\zeta_1$ and $\zeta_2$ is beneficial. 

Our result implies that for a  general objective function, geometric uniformity, captured by the RIP property, can guarantee a benign optimization landscape even when $\delta$ is non-trivially larger than 0. However, this comes with a caveat. In particular, if $\zeta_2 \neq 0$, meaning that the Hessian is affected by the existence of noise, then there is a hard "contribution floor" for the noise reflected by the inequality $\|w\|_2 \leq \frac{1/3-\delta}{\zeta_2}$. If the noise intensity goes beyond this hard limit, no high-probability guarantees can be made in terms of the locations of the local minima. This is expected because if $\zeta_2$ is large, it means that the RIP property satisfied for the noiseless problem cannot enforce any desirable property on the highly noisy problem and the benign optimization landscape is unlikely to hold.

The proof of Theorem \ref{thm:global_local_min} follows from the characterization of the $r$-th singular value of an arbitrary local minimizer $\hat X$. Previous results in the literature successfully upper-bounded the $r$-th singular value of $X$ that are far from the ground truth, which leads to the establishment of a significant escape direction based on its Hessian. The major innovation in the proof of Theorem~\ref{thm:global_local_min} is based on the observation that for every local minimizer $\hat X$, its $r$-th singular value can also be lower-bounded in terms of the smallest eigenvalue of the gradient at $\hat X$, and the RIP constant. Then we adopt some existing techniques to also upper-bound the $r$-th singular value of $\hat X$ to contrast it with the lower-bound. By doing so, we derive necessary conditions on the value of $\|\hat X \hat X^\top - M^*\|$, since the upper-bounds are carefully crafted to include this term. We believe that this new method of lower-bounding the $r$-th singular value of $\hat X$ could open up a new range of possible techniques for analyzing low-rank optimization problems, since it provides important complementary information on $\hat X$. The full proof is lengthy and deferred to Appendix \ref{sec:appendix_global}.


\subsection{When $\delta \geq 1/3$}
\label{sec:location_local}
Although Theorem~\ref{thm:global_local_min} is powerful in the case of $\delta < 1/3$, it does not provide any guarantee when $\delta \geq 1/3$, especially given the fact that $\delta$ is intrinsic to the sensing matrices, which are impossible to change. This is where a local version of the guarantee comes in handy. We only consider the optimization landscape in a region around the ground truth and show that local minimizers are all very close to $M^*$.
\begin{theorem}
	\label{thm:local}
	Assume that the objective function of \eqref{eq:main_noisy_problem} satisfies assumptions 1-3 with $f(M,0)$ satisfying the $\delta$-RIP$_{2r,2r}$ property for a constant $\delta \in [0,1)$. Consider an arbitrary number $\tau \in (0,1-\delta^2)$. Every local minimizer $\hat X \in \mathbb{R}^{n \times r}$ of \eqref{eq:main_noisy_problem} satisfying:
	\begin{equation}
		\label{eq:tau_range}
		\|\hat X \hat X^\top - M^*\|_F \leq \tau \lambda_r(M^*),
	\end{equation}
	will also satisfy the following inequality with probability at least $\mathbb P(\norm{w}_2
 \leq \epsilon)$:
	\begin{equation}
		\label{eq:local_range}
		\|\hat X \hat X^\top - M^*\|_F \leq \frac{\epsilon (1+\delta+\zeta_2 \epsilon) \zeta_1  C(\tau,M^*)}{\sqrt{1-\tau}- \zeta_2 \epsilon -\delta}
	\end{equation}
	for all $\epsilon < \frac{\sqrt{1-\tau}-\delta}{\zeta_2}$, where
	\[
		C(\tau,M^*)=\sqrt{\frac{2(\lambda_1(M^*)+\tau \lambda_r(M^*))}{(1-\tau)\lambda_r(M^*)} }.
	\]
\end{theorem}
The upper bounds in \eqref{eq:tau_range} and \eqref{eq:local_range} define an outer ball and an inner ball centered at the ground truth $M^*$. Theorem~\ref{thm:local} asserts the absence of local minima in the ring between the two balls. As $\epsilon$ goes to $0$, Theorem~\ref{thm:local} states that no spurious local minima exists when $\|\hat X \hat X^\top - M^*\|_F \leq (1-\delta^2) \lambda_r(M^*)$. Therefore, this is a direct generalization of the results in \cite{bi2020global}, which holds only for noiseless objectives. This local theorem allows for the analysis of highly non-convex objectives associated with $\delta$ close to 1. In particular, Theorem~\ref{thm:local} states that even for highly non-convex objectives, the optimization landscape is benign in the vicinity of $M^*$. This means that if a good initial point is selected, local search algorithms can solve this highly non-convex problem and find a satisfactory approximate solution. 

The breakthrough of the proof of this theorem relies on the establishment of Lemma \ref{lem:noisy_foc_nec}, which states that for every local minimizer $\hat X$ of the noisy problem \eqref{eq:main_noisy_problem}, there is a pseudo sensing matrix $\mathbf H$ such that $\hat X$ is an approximate local minimizer of a matrix sensing problem with the sensing operator $\mathbf H$. This serves as the basis of the ensuing proof techniques, which follow the idea of certifying the in-existence of spurious local minima, inspired by \cite{Zhang2021-p,ma2021sharp}. A detailed proof can be found in Appendix \ref{sec:appendix_local}.

\section{CONVERGENCE RATE}

\begin{figure*}[t]
    \centering
    \begin{subfigure}{6.5cm}
    	    \includegraphics[width=\linewidth]{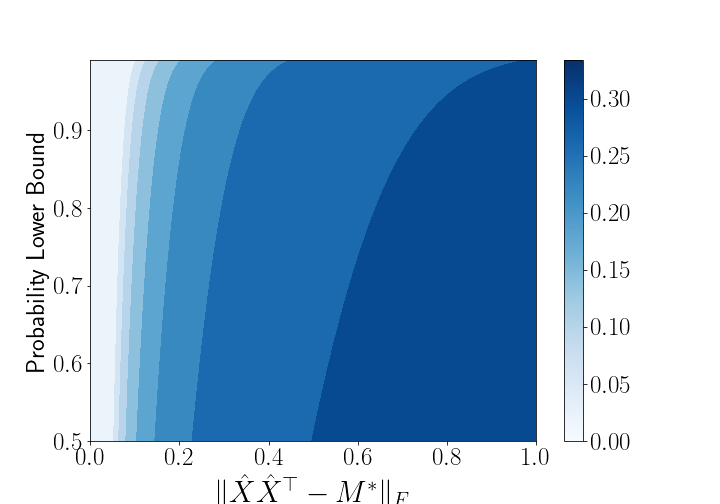}
    \caption{$\delta$ bound in Theorem~\ref{thm:global_local_min}.}
    \end{subfigure}\hspace{2em}
    \begin{subfigure}{6.5cm}
    	    \includegraphics[width=\linewidth]{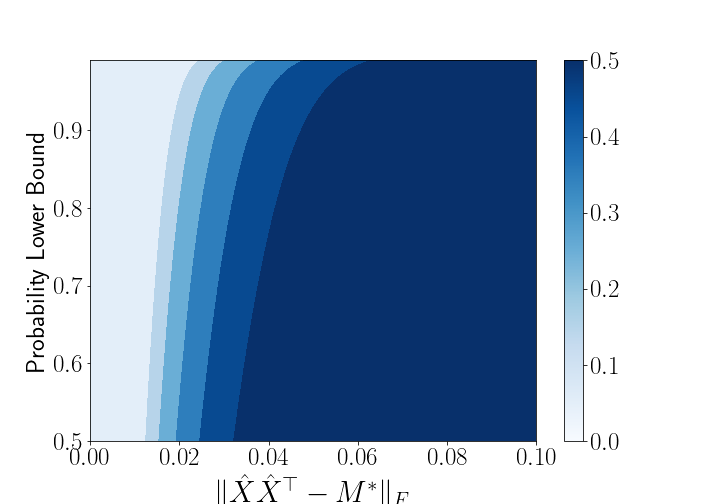}
    \caption{$\delta$ bound in Theorem~\ref{thm:local} with $\tau = 0.1$.}
    \end{subfigure}\vspace{2em}\\
    \begin{subfigure}{6.5cm}
    	    \includegraphics[width=\linewidth]{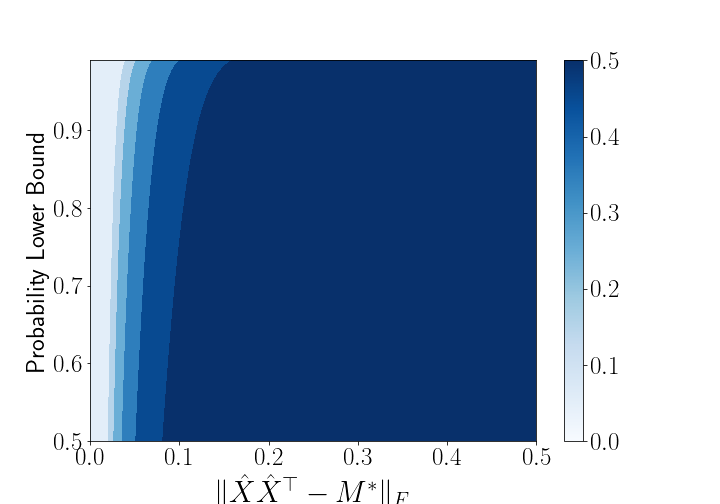}
    \caption{$\delta$ bound in Theorem~\ref{thm:local} with $\tau=0.5$.}
    \end{subfigure}\hspace{2em}
    \begin{subfigure}{6.5cm}
    	    \includegraphics[width=\linewidth]{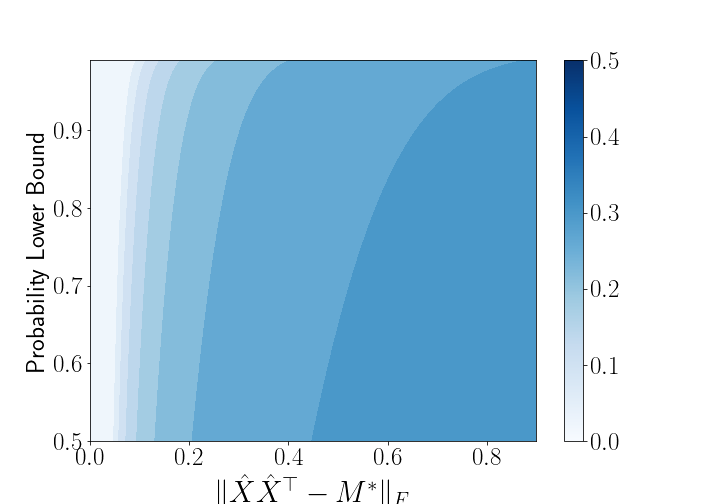}
    \caption{$\delta$ bound in Theorem~\ref{thm:local} with $\tau=0.9$.}
    \end{subfigure}
    \caption{Comparison of the maximum RIP constants $\delta$ allowed by Theorem~\ref{thm:global_local_min} and Theorem~\ref{thm:local} to guarantee a given bound on the distance $\|\hat X\hat X^\top - M^*\|_F$ for an arbitrary local minimizer $\hat X$ satisfying \eqref{eq:tau_range} with a given probability.}
    \label{fig:thm_compare}
\end{figure*}

\subsection{Linear Convergence with good initialization}
\label{sec:convergence_local}

To establish linear convergence for the noisy problem \eqref{eq:main_noisy_problem}, an additional assumption is required:
\begin{description}
	\item[Assumption 4.] There exists a constant $\rho$ such that the gradient of the function $f(\cdot,w)$ with respect to the first argument $M$ is $\rho-$restricted Lipschitz continuous, meaning that:
		\[
			\| \nabla_M f(M,w) - \nabla_M f(M',w) \|_F \leq \rho \|M - M'\|_F
		\]
		for all matrices $M, M' \in \mathbb{R}^{n \times n}$ with $\rk(M) \leq r$ and $\rk(M') \leq r$.
\end{description}
Assumption 4 is critical for the convergence of local search algorithms since otherwise we cannot choose a step size small enough to avoid the constant overshoot of the algorithm. For a standard matrix sensing problem, $\nabla_M f(M,w) = \mathbf{A}\mathbf{A}^\top \vecc(M)+\mathbf{A}^\top w$, hence satisfying Assumption 4 with $\rho = \sigma_{\max}(\mathbf{A}\mathbf{A}^\top)$.

We now present our main result in this section, which states that if the initialization is close enough to $M^w$, then the gradient descent algorithm will reach $M^w$ or a low-rank projection of $M^w$ at a linear rate. Here, $M^w$ is defined to be the unique global minimum of  \eqref{eq:unfactored_noisy_problem} without the rank constraint. Since \eqref{eq:unfactored_noisy_problem} is a strongly convex problem without the rank constraint, $M^w$ always exists and is unique. Given Theorems \ref{thm:global_local_min} and \ref{thm:local}, we can in turn guarantee that $M^w$ is close to $M^*$, showing that the gradient descent algorithm reaches a neighborhood of $M^*$ in a satisfactory rate.
\begin{theorem}
	\label{thm:linear_conv}
	The vanilla gradient descent method applied to \eqref{eq:main_noisy_problem} under Assumptions 1-4 converges to $\projr(M^w)$, the best rank-r approximation of $M^w$, linearly up to a difference $D_r$ if the initial point $X_0$ satisfies:
	\begin{equation}
		\label{eq:pl_init}
		\|X_0 X_0^\top - M^w \|_F < C_w^2(1-\delta-\zeta_2 \epsilon)  -  C_w \sqrt{\frac{1-\delta-\zeta_2 \epsilon}{1+\delta+\zeta_2 \epsilon}} D_r,
	\end{equation}
	meaning that vanilla gradient descent will reach a point $\tilde M$ linearly with $\|\tilde M - \projr(M^w)\|_F \geq D_r$, where
	\[
		D_r = \|M^w - \projr(M^w)\|_F, \ \ C_w = \sqrt{2(\sqrt 2 -1)\sigma_r(M^w)}.
	\]
	The linear convergence is also contingent on the fixed step size $\eta$ satisfying:
	\begin{equation}
			\eta \leq  \left(12 \rho r^{(1/2)}\left(C\sqrt{(1-(\delta+\zeta_2 \epsilon)^2}  + \|M^w\|_F \right)\right)^{-1},	
	\end{equation}
	for all $\epsilon < \frac{1-\delta}{\zeta_2}$ with probability at least $\mathbb P(\norm{w}_2 \leq \epsilon)$, where $C = 2(\sqrt 2 -1)$.

\end{theorem}
The main challenge stemming from the introduction of noise is that the unconstrained global minimum of \eqref{eq:unfactored_noisy_problem} may not necessarily be of rank-$r$ (the rank of $M^*$) anymore. Therefore, since the Monteiro-Burer approach \eqref{eq:main_noisy_problem} can only search over matrices of rank at most $r$, we can only guarantee the convergence of any algorithm with respect to a rank-$r$ matrix, which for our purpose we chose to be $\projr(M^w)$. Thus, the radius of linear convergence depends on $D_r$, a constant quantifying how close $M^w$ is to a rank-$r$ matrix. In the special case that $M^w$ is of rank at most $r$, $D_r$ becomes 0 and our Theorem can be simplified. We summarize this special case via the following assumption:
\begin{description}
	\item[Assumption 5.] The objective function $f(\cdot,w)$ of \eqref{eq:unfactored_noisy_problem} has a first-order critical point $M^w$ for every $w$ such that it is symmetric, positive semidefinite, and $\rk(M^w) \leq r$.
\end{description}
Assumption 5 may not hold in general, but for specific problems, such as \eqref{eq:linear_noisy_problem}, this assumption is satisfied if the set $\{\mathbf{A}\vecc(N-M^*) \ | \ \rk(N)\leq r \}$ spans $\mathbb{R}^m$. This is highly likely since $m \ll n^2$. According to Proposition 1 in \cite{zhu2018global}, if Assumption 5 is met, $M^w$ is the global minimum of \eqref{eq:unfactored_noisy_problem}. With this assumption, we can now introduce a useful Corollary:
 \begin{corollary}
 	\label{thm:linear_conv_coro}
	The vanilla gradient descent method applied to \eqref{eq:main_noisy_problem} under Assumptions 1-5 converges to $M^w$ linearly if the initial point $X_0$ satisfies:
	\begin{equation}
		\label{eq:pl_init}
		\|X_0 X_0^\top - M^w \|_F < 2(\sqrt 2 -1) (1-\delta-\zeta_2 \epsilon) \sigma_r(M^w),	\end{equation}
	with fixed step size $\eta$ satisfying:
	\begin{equation}
			\eta \leq  \left(12 \rho r^{(1/2)} \left(C\sqrt{(1-(\delta+\zeta_2 \epsilon)^2}  + \|M^w\|_F \right)\right)^{-1},	
	\end{equation}
	for all $\epsilon < \frac{1-\delta}{\zeta_2}$ with probability at least $\mathbb P(\norm{w}_2 \leq \epsilon)$, where $C = 2(\sqrt 2 -1)$.
 \end{corollary}
Prior to this theorem, it was possible to establish a linear convergence using the existing literature only when Assumption 5 holds and $\delta < 1/7$. Now, Theorem \ref{thm:linear_conv} allows for having an arbitrary $\delta$, and generalized the guarantee to cases where Assumption 5 does not hold. Theorem \ref{thm:linear_conv} further implies that even starting from an arbitrary initial point, the gradient descent algorithm has a linear convergence in the final phase, given that the step size is small enough and that the noise intensity is not high. This further implies that if a linear convergence is not observed, the user could decrease the step size until a linear convergence is established. This is confirmed empirically in Section~\ref{sec:num_conv}.

Theorem \ref{thm:linear_conv} is inspired by the observation that since we only search on a low-rank manifold, we may never really reach $M^w$ (even in the asymptotic regime), thus by constraining the search space away from $M^w$, linear convergence can be established. The full proof is deferred to Appendix \ref{sec:appendix_linear_cov}.

\subsection{Strict Saddle Property}
\label{sec:strict_saddle}
When $\delta < 1/3$, the noisy problem \eqref{eq:main_noisy_problem} exhibits the strict saddle property, meaning that all approximate second-order critical points are close to the global optimum of the optimization problem with high probability:
\begin{theorem}
	\label{thm:strict_saddle}
	Suppose that the objective function of \eqref{eq:main_noisy_problem} satisfies assumptions 1-3 with a $\delta$-RIP$_{2r,2r}$ constant of $\delta < 1/3$ in the noiseless case. Consider the ground truth solution $M^*$ which is of rank $r$. For a given constant $\alpha > 0$, there exists a finite constant $\xi > 0$ such that at least one of the three following conditions holds for any $X \in \mathbb{R}^{n \times r}$:
	\begin{align*}
		&\dist(X,M^*) \leq \alpha, \ \|\nabla_X h(X,w)\|_F \geq \xi, \\
		 &\lambda_{\min}(\nabla^2_X h(X,w)) \leq -2\xi,
	\end{align*}
	with probability at least $\mathbb P(\norm{w}_2
 \leq \frac{1/3-\delta}{\zeta_2+2\zeta_\alpha/3})$, where $\zeta_\alpha \coloneqq \zeta_1/(\sqrt{2(\sqrt 2-1)} (\sigma_r(M^*))^{1/2} \alpha)$. 
\end{theorem}
The significance of the establishment of the strict saddle property is that one can find an approximate local minimum in polynomial time. The perturbed gradient descent algorithm presented in \cite{JGNK2017} serves as one of the algorithms achieving this goal. Coupled with Theorem~\ref{thm:linear_conv}, it means that we could reach $M^w$ with an arbitrary accuracy in polynomial time via a random initialization, which is also known to be close to $M^*$ according to Theorems \ref{thm:global_local_min} and \ref{thm:local}.

The proof of this theorem is similar to that of Theorem 7 of \cite{zhang2021general}, and we highlight the key differences in Appendix \ref{sec:appendix_strict_saddle} to illustrate how Theorem \ref{thm:strict_saddle} can be proved.

\begin{figure*}[t]
    \centering
    \begin{subfigure}{7cm}
    	    \includegraphics[width=\linewidth]{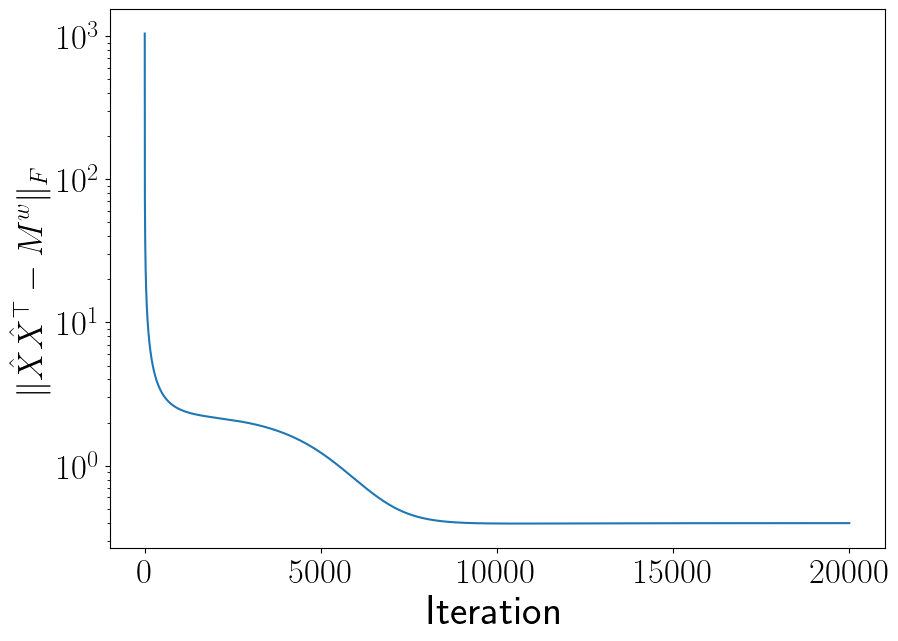}
    \caption{Convergence rate when step size is 0.001.}
    \end{subfigure} \hspace{1em}
    \begin{subfigure}{7cm}
    	    \includegraphics[width=\linewidth]{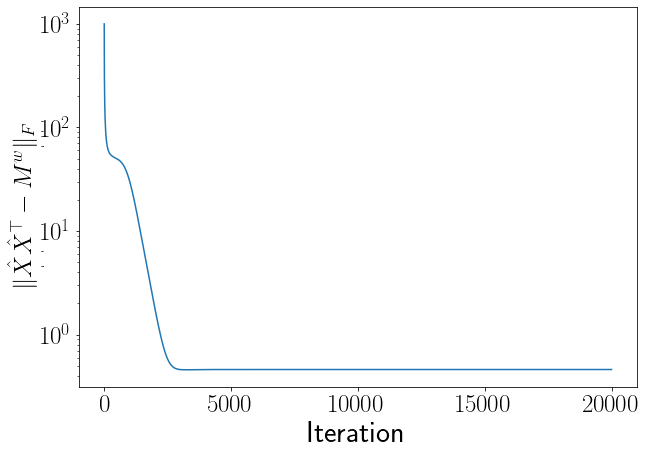}
    \caption{Convergence rate when step size is 0.0002.}
    \end{subfigure}
    \caption{The distance to $M^w$ versus iterations for gradient descent with random initialization.}
    \label{fig:convergence}
\end{figure*}

\section{NUMERICAL ILLUSTRATION}
In this section, we provide a concrete example to the results derived above\footnote{Code used to produce the results in this section can be found here: https://github.com/anonpapersbm/Noisy-Low-rank-Matrix-Optimization}. We empirically study the proximity of an arbitrary local minimizer $\hat X$ of \eqref{eq:main_noisy_problem} to its ground truth solution in terms of $\|\hat X \hat X^\top -M^*\|_F$, and analyze the effect of the step size on the convergence rate.

Assume that $w \in \mathbb{R}^m$ is a $0.05/\sqrt{m}$-sub-Gaussian vector. According to Lemma 1 in \cite{jin2019short}, this choice of $w$ satisfies:
\[
	1- 2\mathrm e^{-\frac{\epsilon^2}{16m \sigma^2}} \leq \mathbb{P}(\|w\|_2 \leq \epsilon).
\]
with $\sigma = 0.05$. We refer to the RHS of the above equation as the probability lower-bound since it says that the event $\|w\|_2 \leq \epsilon$ will happen with probability at least that number.

\subsection{Quality of Local Minima}
We consider the problem of 1-bit Matrix Completion, which is a low-rank matrix optimization problem that naturally arises in recommendation systems with binary inputs \citep{davenport20141,ghadermarzy2018learning}. 

The objective of this 1-bit Matrix Completion problem is:
\begin{equation}
	\label{eq:1bit}
	f(M,w) = -\sum_{i=1}^n \sum_{j=1}^n ((y_{ij}+w_{ij})M_{ij} - \log(1+\exp(M_{ij})))
\end{equation}
where $M_{ij}$ is the $(i,j)^{\text{th}}$ component of $M$ and $y_{ij} \in [0,1]$ is a percentage-wise observation of $M_{ij}$. Since $y_{ij}$ are empirical observations, they could very much be subject to random corruptions, which we explicitly represent by $w \in \mathbb{R}^{m}$, with $m=n^2$. It is straightforward to verify that \eqref{eq:1bit} satisfies the assumptions outlined in Section \ref{sec:assump}, with $\zeta_1 = 1$ and $\zeta_2 =0$.

The work \cite{bi2020global} shows that for \eqref{eq:1bit}, the function $\gamma f(M,0)$ exhibits the $\delta$-RIP$_{2r,2r}$ property for some constant $\gamma$ over the neighborhood $\|\hat X \hat X^\top - M^*\|_F \leq R$ for a small $R$. We choose $M^*$ such that $\lambda_r(M^*) = R$. Therefore, we can use the framework proposed in this paper to analyze the quality of the local minima of \eqref{eq:1bit} under random perturbation.

In Figure~\ref{fig:thm_compare}, we numerically demonstrate and compare the bounds given in Theorem~\ref{thm:global_local_min} and Theorem~\ref{thm:local}, for the parameters $n=40$, and $r=5$. We assume $w$ comes from the sub-Gaussian distribution described above with $\sigma=0.05$. The x-axis shows the maximum distance between an arbitrary local minimum $\hat X$ and the ground truth, and unit for the x-axis is $\lambda_r(M^*)$. The y-axis delineates the probability lower bound, which describes a lower-bound on the probability that the event will happen. The contour plot itself shows the maximum $\delta$ that is necessary to guarantee $\hat X$ to be in the range of $\|\hat X \hat X^\top - M^*\|_F \leq \xi$ with $\xi$ specified on the x-axis, with probability greater than the value specified on the y-axis. Figure~\ref{fig:thm_compare} shows that as $\tau$ becomes smaller, for the same set of $(x,y)$ values, the necessary value of $\delta$ becomes larger. This means that if the prior information on $\hat X$ is strong, meaning that it is known to lie within a neighborhood of the ground truth, then the local minima are tightly centered around the correct solution with a high probability. Moreover, the global bound is generally looser than that of the local version when $\tau$ is small, because it only applies to cases when $\delta < 1/3$, but when $\tau$ is large, the global bound could be better even with the same $\delta$, as evident when comparing subfigures (a) and (d) in Figure~\ref{fig:thm_compare}.

Readers can refer to Appendix \ref{sec:figures} for more plots regarding the interplay of $\zeta_1,\zeta_2$ and $\sigma$ values in Theorem \ref{thm:global_local_min}. These values may represent a wide range of different objectives(each objective is characterized by its $\zeta_1$ and $\zeta_2$ values) and different noise patterns (characterized by $\sigma$ value since many known distributions are sub-Gaussian).

\subsection{Convergence Rate}
\label{sec:num_conv}

In this section, we demonstrate the convergence rate of the vanilla gradient descent algorithm applied to an instance of \eqref{eq:linear_noisy_problem} satisfying Assumptions 1-5 with $n=40$, $m=190$, $r=5$. The matrix $\mathbf A$ used here makes the objective function satisfy $0.42$-RIP$_{2r,2r}$. We also assume that $\lambda_1(M^*)=1.5$ and $\lambda_r(M^*)=1$. Note that \eqref{eq:linear_noisy_problem} meets our assumptions with $\rho = \|\mathbf A \|^2_2, \zeta_1 = \|\mathbf A\|_2$, and $ \zeta_2=0$. We aim to show how the step size affects the convergence rate, and corroborate the theoretical results in Theorem \ref{thm:linear_conv}. Note since Assumption 5 is satisfied, the algorithm will converge to $M^w$ directly.

In Figure~\ref{fig:convergence}, we choose two different step sizes, namely 0.001 and 0.0002, and start from random initialization. It can be observed that in the case of the larger step size, there is a region of plateauing in which the gradient descent algorithm makes little progress, while the smaller step size exhibits a linear convergence around iterations 500-2500 even after the initial phase of a fast descent. This result is in accordance with Corollary~\ref{thm:linear_conv_coro}, which states that for a small enough step size, the gradient descent algorithm will achieve a linear convergence in a neighborhood of the global minimum.

\section{CONCLUSION}
In this work, we proposed a unified, yet general framework to analyze the global and local optimization landscapes of a class of noisy low-rank matrix optimization problems. We showed that regardless of the distribution from which the random noise is sampled, if the noiseless objective satisfies RIP, then there are mathematical guarantees on the locations of local minima and the convergence rate. This means that even for general objectives, geometric uniformity can compensate for random corruption. This paper significantly extends the existing results in the literature on this general problem, and offers new techniques and insights that can be used to study other noisy low-rank optimization problems.

\section{ACKNOWLEDGEMENTS}
This work was supported by grants from ONR and NSF.


\bibliography{references.bib}

\appendix
\onecolumn
\newpage

\section{ASYMMETRIC CASE}\label{sec:appendix_asymmetric}
Although this paper focuses on the symmetric case, meaning that $M$ is assumed to be symmetric and positive semidefinite, our analysis techniques are all valid for the case where $M \in \RR^{n \times m}$ for arbitrary numbers $m$ and $n$. To explain this generalization as per \cite{tu2016low}, we first need to deal with the redundancy of global optima induced by the asymmetry. This can be achieved by solving the following optimization problem with a regularization term instead of \eqref{eq:main_noisy_problem}:
\begin{equation}
	\label{eq:asymm_prob}
	\min_{U \in \mathbb{R}^{n \times r}, \ V \in \RR^{m \times r} }f(UV^\top,w)+\frac{\phi}{4} \|U^\top U - V^\top V\|_F^2.
\end{equation}
where $\phi$ is an arbitrary penalization constant. As per \cite{bi2021local}, solving \eqref{eq:asymm_prob} is equivalent to:
\begin{equation}
	\label{eq:asymm_prob_converted}
	\min_{X \in \RR^{(n+m) \times r}} f_a(XX^\top,w)
\end{equation}
where $X = \begin{bmatrix}
	U^\top & V^\top
\end{bmatrix}^\top \in \RR^{(n+m) \times r}$ and the function $f_a(\cdot,w): \RR^{(n+m) \times (n+m)} \mapsto \RR$ satisfies:
\begin{equation*}
	\begin{aligned}
		f_a(\begin{bmatrix}
		P_{11} & P_{12} \\ P_{21} & P_{22}
	\end{bmatrix},w) = \frac{f(P_{11},w)+ f(P_{22},w)}{2} + \\
	\frac{\phi}{4}(\|P_{11}\|^2_F+\|P_{22}\|^2_F-\|P_{12}\|^2_F-\|P_{21}\|^2_F)
	\end{aligned}
\end{equation*}
where $P_{11} \in \mathbb{R}^{n \times n}, \quad P_{12} \in \mathbb{R}^{n \times m}, \quad P_{21} \in \mathbb{R}^{m \times n}, \quad P_{22} \in \mathbb{R}^{m \times m}$ are just partitioned blocks of $XX^\top$ of appropriate dimensions corresponding to $UU^\top, UV^\top, VU^\top, VV^\top$, respectively. The equivalence between the asymmetric problem and its symmetric counterpart \eqref{eq:asymm_prob_converted} implies that the results of this paper obtained for \eqref{eq:main_noisy_problem} can be restated for the original asymmetric problem.

\section{OPTIMALITY CONDITIONS}
Before diving into the results, we establish some optimality conditions for local minima and global minima, which we will use extensively in the ensuing sections.

First, we make the following assumption without loss of generality:
\begin{assumption}
	Assume that $\nabla_M f(M,w)$ is symmetric for every $M \in \RR^{n \times n}$,
\end{assumption}
This assumption always holds since otherwise we could simply optimize for $(f(M,w)+f(M^\top,w))/2$ instead.

The parameter $q$ is used to represent $\|w\|_2$ in the following proofs for the sake of notational simplicity.

We also make the following standard assumption:
\begin{assumption}
	\label{assump:critical}
	The objective function $f(\cdot,0)$ of \eqref{eq:unfactored_noisy_problem} has a first-order critical point $M^*$ such that it is symmetric, positive semi-definite, and $\rk(M^*) \leq r$.
\end{assumption} 
This assumptions states that the objective in \eqref{eq:unfactored_noisy_problem} can indeed recover the ground truth low rank matrix $M^*$ when solved to global optimality. Otherwise solving for \eqref{eq:unfactored_noisy_problem} under low rank constraint will be meaningless.

Note that 
\begin{equation}
	\label{eq:gt_nabla_zero}
	\nabla_M f(M^*,0) = 0
\end{equation}
is a consequence of Assumption \ref{assump:critical} due to Proposition 1 in \cite{zhu2018global}. This proposition also implies that $M^*$ is the unique global minimum of \eqref{eq:unfactored_noisy_problem}.

Next, we derive necessary and sufficient conditions for first- and second-order critical points of \eqref{eq:main_noisy_problem}:
\begin{lemma}
	\label{lem:sosp}
	A matrix $\hat X \in \mathbb{R}^{n \times r}$  is a second-order critical point of problem \eqref{eq:main_noisy_problem} if and only if
	\begin{equation}
		\label{eq:foc}
		\nabla_M f(\hat X \hat X^\top,w) \hat X = 0
	\end{equation}
	and
	\begin{equation}
		\label{eq:soc}
	\begin{aligned}
		&2 \langle \nabla_M f(\hat X \hat X^\top,w), UU^\top \rangle + [\nabla^2 f(\hat X \hat X^\top,w)](\hat X U^\top + U \hat X^\top,\hat X U^\top + U \hat X^\top) \geq 0
	\end{aligned}
	\end{equation}
	for all $U \in \mathbb{R}^{n \times r}$. Furthermore $\hat X \in \mathbb{R}^{n \times r}$ is a first-order critical point if and only if it satisfies \eqref{eq:foc}.
\end{lemma}
Note that for the unconstrained optimization problem \eqref{eq:main_noisy_problem}, a second-order critical point is a local minima. This lemma can be proved by simply deriving the gradient and Hessian of the unconstrained problem \eqref{eq:main_noisy_problem}. Thus, the proof is omitted for brevity. 

\section{MISSING PROOFS}

\subsection{Proofs of Section~\ref{sec:location_global}}
\label{sec:appendix_global}

\begin{lemma}
	\label{lem:lb_lamb_r}
	If $\hat X$ is a local minimum of \eqref{eq:main_noisy_problem} with $\hat M = \hat X \hat X^\top$, then
	\begin{equation}
		\label{eq:lb_lamb_r}
		\lambda_r^2(\hat M) \geq \frac{G^2}{(1+\delta+\zeta_2 q)^2}
	\end{equation}
	where $G = -\lambda_{\min}(\nabla_M f(\hat M, w))$.
\end{lemma}

\begin{proof}[Proof of Lemma \ref{lem:lb_lamb_r}]
	First consider the case where $\rk(\hat M) = r$. Under this assumption, consider the singular value decomposition (SVD) of $\hat M$:
	\[
		\hat M = \sum_{i=1}^{r} \sigma_i u_i u_i^\top,
	\]
	where $\sigma_i$'s are eigenvalues and $u_i$'s are unit eigenvectors. Let $u_G$ be a unit eigenvector of $\nabla f(\hat M, w)$ such that $u_G^\top \nabla f(\hat M, w) u_G = -G$. Furthermore, for a constant $p \in [0,1]$, define:
	\begin{equation*}
		\begin{aligned}
			M_p = \sum_{i=1}^{r-1} \sigma_i u_i u_i^\top + \sigma_r (p u_G + \sqrt{1-p^2} u_r)(p u_G + \sqrt{1-p^2} u_r)^\top.
		\end{aligned}
	\end{equation*}
	One can write:
	\begin{equation*}
		\begin{aligned}
			\langle \nabla_M f(\hat M, w), M_p - \hat M \rangle &= \langle \nabla_M f(\hat M, w),  \sigma_r p^2 u_G u_G^\top \rangle \\
			&= -G p^2 \sigma_r.
		\end{aligned}
	\end{equation*}
	since $ \nabla f(\hat M, w) u_i = \nabla f(\hat M, w)^\top u_i = 0 \ \forall i \in \{1,\dots,r\}$. This is because $\hat X$ is a local minimum, and \eqref{eq:foc} is a necessary condition according to Lemma \ref{lem:sosp}. We could choose a SVD of $\hat M$ such that:
	\[
		\hat X = \begin{bmatrix}
			\sigma_1^{1/2} u_1 & \sigma_2^{1/2} u_2 & \dots & \sigma_r^{1/2} u_r
		\end{bmatrix}.
	\]
	Now, we expand the term $\|M_p - M\|^2_F$:
	\begin{equation*}
		\begin{aligned}
			\|M_p - \hat M\|^2_F = &\sigma_r^2 \tr \left((p^2 u_G u_G^\top + p\sqrt{1-p^2} u_G u_r^\top + p \sqrt{1-p^2} u_r u_G^\top -p^2 u_r u_r^\top)^2 \right) \\
			= & \sigma_r^2 (p^4 + p^2(1-p^2) + p^2 (1-p^2) + p^4) \\
			=& 2 \sigma_r^2 p^2.
		\end{aligned}
	\end{equation*}
	where the second equality follows from the fact that $u_G^\top u_i = 0 \ \ \forall i \in \{1,\dots,r\}$. This is due to the fact that
	\[
		u_G^\top u_i = \left(\frac{-1}{G} \nabla_M f(\hat M, w) u_G \right)^\top u_i = 0.
	\]
	This means that $\langle \nabla_M f(\hat M, w), M_p - \hat M \rangle = - \frac{G}{2 \sigma_r} \|M_p - \hat M\|^2_F$.
	Next, we proceed with the proof by contradiction. First, assume that $G > \sigma_r (1+\delta + \zeta_2 q)$. Then, there exists a small constant $c$ such that:
	\begin{equation}
		\label{eq:foe_ub}
		\langle \nabla_M f(\hat M, w), M_p - \hat M \rangle < -\frac{(1+\delta+\zeta_2 q)+c}{2} \|M_p - \hat M\|^2_F.
	\end{equation}
	Second, combining the Taylor expansion of $f(M,w)$ in terms of $M$ at the point $\hat M$ with the mean-value theorem gives:
	\begin{equation*}
		\begin{aligned}
			f(M_p,w) = &f(\hat M,w) + \langle \nabla_M f(\hat M,w), M_p - \hat M \rangle + \\
			&\frac{1}{2}[\nabla^2 f(\tilde M,w)](M_p - \hat M,M_p - \hat M),
		\end{aligned}
	\end{equation*}
	for some matrix $\tilde M$ that is a convex combination of $M_p$ and $\hat M$. Due to the RIP assumption and \eqref{eq:noise_perturb_hess}, we have:
	\begin{equation}
		\label{eq:foe_lb}
		\begin{aligned}
			f(M_p,w) \leq &f(\hat M,w) + \langle \nabla_M f(\hat M,w), M_p - \hat M \rangle + \\
			&\frac{1}{2}[(1+\delta+\zeta_2 q)+c] \|M_p - \hat M\|_F^2,
		\end{aligned}
	\end{equation}
	for the same small constant $c$ used above. Therefore, by combining \eqref{eq:foe_ub} and \eqref{eq:foe_lb}, we have:
	\[
		f(M_p,w) < f(\hat M,w),
	\]
	which is a contradiction due to the fact that $\hat X$ is a local minimum since we can adjust $p$ to make $M_p$ arbitrarily close to $\hat M = \hat X \hat X^\top $ and that $M_p$ is a positive semidefinite matrix of rank $r$. This further leads to the conclusion that $G \leq \sigma_r (1+\delta + \zeta_2 q)$, consequently leading to \eqref{eq:lb_lamb_r}.
	
	Then consider the case where $\rk(\hat M) < r$. By \cite{ha2020equivalence}, we know that $\hat M$ is a critical point of \eqref{eq:unfactored_noisy_problem}, meaning that if $\rk(\hat M) < r$, $\nabla_M f(\hat M,w) = 0$. Therefore $G = 0$ and \eqref{eq:lb_lamb_r} is trivially satisfied since $\lambda_r(\hat M) = 0$.
\end{proof}

\begin{proof}[Proof of Theorem \ref{thm:global_local_min}]
	Define $\hat M \coloneqq \hat X \hat X^\top$ and
	\begin{equation}
		\bar M \coloneqq \hat M - \frac{1}{1+\delta+\zeta_2 q} \nabla_M f(\hat M, w).
	\end{equation}
	Additionally, define $\phi(\cdot)$ as
	\[
		\phi(M) \coloneqq \langle \nabla_M f(\hat M,w) , M - \hat M \rangle + \frac{1+\delta+\zeta_2 q}{2} \| M - \hat M\|^2_F.
	\]
	Now,
	\begin{equation*}
		\begin{aligned}
			\frac{1+\delta+\zeta_2 q}{2} \|M-\bar M\|^2_F &= \frac{1+\delta+\zeta_2 q}{2} \|M- \hat M + \frac{1}{1+\delta+\zeta_2 q} \nabla_M f(\hat M, w) \|^2_F \\
			&=  \frac{1+\delta+\zeta_2 q}{2} \| M - \hat M\|^2_F + \langle \nabla_M f(\hat M,w) , M - \hat M \rangle + \frac{1}{(1+\delta+\zeta_2 q)^2} \| \nabla_M f(\hat M,w) \|^2_F \\
			&= \phi(M) + \text{constant with respect to} \ M.
		\end{aligned}
	\end{equation*}
	Define $\projr(M)$ of an arbitrary matrix $M$ to be the projection of $M$ on a low-rank manifold of rank at most $r$:
	\[
		\projr(M) = \argmin_{M_r \in \mathcal{M}} \| M_r -M\|_F, \qquad  \mathcal{M} \coloneqq \{M \in \mathbb{S}^{n \times n} \vert \rk(M) \leq r, M \succeq 0 \}
	\]
	Then by the Eckart-Young-Mirsky Theorem, $\phi(\projr(\bar M))$ achieves the minimum value of the function $\phi(\cdot)$ over all matrices of rank at most $r$. Therefore,
	\begin{equation}
	\label{eq:thm_1_help_1}
		\begin{aligned}
			-\phi(\projr(\bar M)) \geq -\phi(M^*) = \langle \nabla_M f(\hat M,w), \hat M - M^* \rangle - \frac{1+\delta+\zeta_2 q}{2} \|M^*-\hat M\|^2_F.
		\end{aligned}
	\end{equation}
	
	Next, we apply the Taylor expansion to $f(M,w)$ at $\hat M$ and combine it with the RIP property to obtain
	\begin{equation}
		\label{eq:thm_1_help_2}
		\begin{aligned}
			f(M^*,w) \geq f(\hat M,w) + \langle \nabla_M f(\hat M,w), M^* - \hat M \rangle + \frac{1-\delta-\zeta_2 q}{2} \|M^* - \hat M\|^2_F.
		\end{aligned}
	\end{equation}
	Additionally, by expanding at $M^*$, we can also write:
	\begin{equation}
	\label{eq:thm_1_help_3}
		\begin{aligned}
			f(\hat M,w) - f(M^*,w) &\geq \langle \nabla_M f(M^*,w), \hat M - M^* \rangle + \frac{1-\delta-\zeta_2 q}{2} \|\hat M - M^*\|^2_F \\
			 &\geq \frac{1-\delta-\zeta_2 q}{2} \|\hat M - M^*\|^2_F - \zeta_1 q \|\hat M - M^*\|_F
		\end{aligned}
	\end{equation}
	where the second inequality follows from \eqref{eq:noise_perturb_grad} and the fact that $M^*$ is the ground truth. Substituting \eqref{eq:thm_1_help_2} into \eqref{eq:thm_1_help_1} gives:
	\begin{equation*}
		\begin{aligned}
			-\phi(\projr(\bar M)) \geq f(\hat M,w) - f(M^*,w) - (\delta + \zeta_2 q) \|\hat M - M^*\|^2_F,
		\end{aligned}
	\end{equation*}
	and a further substitution of \eqref{eq:thm_1_help_3} into the above equation gives:
	\begin{equation}
		\label{eq:thm1_help_4}
		\begin{aligned}
			-\phi(\projr(\bar M)) \geq  \frac{1-3\delta - 3 \zeta_2 q}{2} \|\hat M - M^*\|^2_F - \zeta_1 q \|\hat M - M^*\|_F.
		\end{aligned}
	\end{equation}
	We denote 
	\[
		L \coloneqq \frac{1-3\delta - 3 \zeta_2 q}{2} \|\hat M - M^*\|^2_F - \zeta_1 q \|\hat M - M^*\|_F.
	\]
	Next, for the notational simplicity of the ensuing sections, define:
	\[
		N \coloneqq -\frac{1}{1+\delta+\zeta_2 q} \nabla f(\hat M,w),
	\]
	implying that $\bar M = \hat M + N$. Then,
	\begin{equation*}
		\begin{aligned}
			-\phi(\projr(\bar M)) &= (1+\delta+\zeta_2 q) \langle N, \projr(\hat M+N) - \hat M \rangle - \frac{1+\delta+\zeta_2 q}{2} \| \projr(\hat M+N) - \hat M \|^2_F \\
			 &=  \frac{1+\delta+\zeta_2 q}{2} (\|N\|^2_F - \|\hat M+N-\projr(\hat M+N)\|^2_F) \\
			 &=  \frac{1+\delta+\zeta_2 q}{2} (\|N\|^2_F - \|\hat M+N\|^2_F+\|\projr(\hat M+N)\|^2_F).
		\end{aligned}
	\end{equation*}
	Since $\hat X$ is a local minimizer of \eqref{eq:main_noisy_problem}, it must be a first-order critical point. Therefore, \eqref{eq:foc} holds true, meaning that $\hat M$ and $N$ have orthogonal column/row spaces, leading to $\|\hat M+N\|^2_F = \|\hat M\|^2_F + \|N\|^2_F$.
	
	Furthermore, due to the orthogonal nature of $\hat M$ and $N$, $\|\projr(\hat M+N)\|^2_F$ is simply the sum of the squares of the maximal $r$ eigenvalues of $\hat M$ and $N$ combined, which we assume to be $\lambda_i(\hat M), i \in \{1,\dots,k\}$ and $\lambda_i(N), i \in \{1,\dots,r-k\}$. Therefore,
	\[
		\|\projr(\hat M+N)\|^2_F = \sum_{i=1}^{k} \lambda_i(\hat M)^2 + \sum_{i=1}^{r-k} \lambda_i(N)^2.
	\]
	Subsequently,
	\begin{equation*}
		\begin{aligned}
		 -\phi(\projr(\bar M)) &= \frac{1+\delta+\zeta_2 q}{2} (- \sum_{i=1}^r \lambda_i(\hat M)^2 + \sum_{i=1}^{k} \lambda_i(\hat M)^2 + \sum_{i=1}^{r-k} \lambda_i(N)^2) \\
			&= \frac{1+\delta+\zeta_2 q}{2}  (-  \sum_{i=k+1}^{r} \lambda_i(\hat M)^2 + \sum_{i=1}^{r-k} \lambda_i(N)^2) \\
			&\leq  \frac{1+\delta+\zeta_2 q}{2} (- (r-k) \lambda_r^2(\hat M) + (r-k) \lambda_{\max}^2(N)).
		\end{aligned}
	\end{equation*}
	Then invoking \eqref{eq:thm1_help_4} gives:
	\begin{equation}
		\label{eq:thm1_help_5}
		(r-k) \lambda_r^2(\hat M) \leq - \frac{2 L}{1+\delta+\zeta_2 q} + (r-k) \frac{G^2}{(1+\delta+\zeta_2 q)^2},
	\end{equation}
	where $G = -\lambda_{\min}(\nabla f(\hat M, w))$.
	
	First, assume that $k < r$ and $\lambda_r(\hat M) >0$. We have
	\begin{equation}
		\label{eq:ub_lambda_r}
		\lambda_r^2(\hat M) \leq \frac{G^2}{(1+\delta+\zeta_2 q)^2} - \frac{2}{1+\delta+\zeta_2 q} \frac{L}{r-k},
	\end{equation}
	Now, recall Lemma \ref{lem:lb_lamb_r}, which also holds for all local minimizers $\hat X$. A necessary and sufficient condition for both Lemma  \ref{lem:lb_lamb_r} and \eqref{eq:ub_lambda_r} to hold is that:
	\begin{equation}
		\label{eq:l_neg}
		L \leq 0,
	\end{equation}
	subsequently meaning that,
	\begin{equation*}
		\begin{aligned}
			(1-3\delta-3\zeta_2 q)\|\hat M - M^*\|^2_F - 2 \zeta_1 q \|\hat M - M^*\|_F \leq 0
		\end{aligned}
	\end{equation*}
	which directly gives \eqref{eq:global_local_min_range} after simple rearrangements.
	
	In the case that $k=r$ or $\lambda_r(\hat M)=0$, \eqref{eq:thm1_help_5} reduces to \eqref{eq:l_neg} as well, leading to the same result presented in \eqref{eq:global_local_min_range}.
	
\end{proof}

\subsection{Proof of Section~\ref{sec:location_local}}
\label{sec:appendix_local}
Given a matrix $\hat X$, we aim to find the smallest $\delta$ such that there is an instance of the problem with this RIP constant for which $\hat X$ is a local minimizer that is not associated with the ground truth. For notational convenience, we denote this optimal value as $\delta^*(\hat X)$. Namely, $\delta^*(\hat X)$ is the optimal value to the following optimization problem:
\begin{equation}\label{eq:intro_lmi}
\begin{aligned}
\min_{\delta,f(\cdot,w)} \quad & \delta \\
\text{s.t.} \quad & \text{$\hat{X}$ is a local minimizer of $f(\cdot,w)$} , \\
& \text{$f(\cdot,0)$ satisfies the $\delta$-$\RIP_{2r}$ property}.
\end{aligned}
\end{equation}
By the above optimization problem, we know that $\delta \geq \delta^*(\hat X)$ for all local minimizers $\hat X$ of $f(\cdot,w)$, where $\delta$ is the best RIP constant of the problem. Since \eqref{eq:intro_lmi} is difficult to analyze, we replace its two constraints with some necessary conditions, thus forming a relaxation of the original problem with its optimal value being a lower bound on $\delta^*(\hat X)$.

To find a necessary condition replacing the two constraints, we introduce the following lemma. This is the first lemma that captures the necessary conditions of a critical point of \eqref{eq:main_noisy_problem}, a problem where random noise is considered.
\begin{lemma}
	\label{lem:noisy_foc_nec}
	Assume that the objective function $f(M,w)$ of \eqref{eq:main_noisy_problem} satisfies all assumptions in Section \ref{sec:assump}, and that $\hat X$ is a first-order critical point of \eqref{eq:main_noisy_problem}. Then, $\hat X$ must satisfy the following conditions for some symmetric matrix $\mathbf{H} \in \mathbb{R}^{n^2 \times n^2}$:
	\begin{enumerate}
		\item $\|\hat{\mathbf X}^\top \mathbf{H} \mathbf{e} \| \leq 2 \zeta_1 q \|\hat X\|_2$
		\item $\mathbf{H}$ satisfies the $(\delta+\zeta_2 q)$-RIP$_{2r,2r}$ property, which means that the inequality
		\begin{equation}
		\label{eq:foc_nec_lem_rip}
			(1-\delta-\zeta_2 q) \|M\|^2_F \leq \mathbf m^\top \mathbf{H} \mathbf m \leq (1+\delta+\zeta_2 q)  \|M\|^2_F 
		\end{equation}
		holds for every matrix $M \in \mathbb{R}^{n \times n}$ with $\rk(M)\leq 2r$, where  $\mathbf m = \vecc(M)$ and $\mathbf{e} = \vecc(\hat X\hat X^\top-M^*)$. $\mathbf{\hat X}$ is defined as per Section \ref{sec:notation}.
	\end{enumerate}
\end{lemma}
Given Lemma \ref{lem:noisy_foc_nec}, we can obtain a relaxation of problem \eqref{eq:intro_lmi}, namely the following optimization problem:
	\begin{equation}\label{eq:local_lmi_delta_2r}
\begin{aligned}
\min_{\delta,\mathbf H} \quad & \delta \\
\st \quad & \norm{\hat{\mathbf X}^\top \mathbf H \mathbf e} \leq 2 \zeta_1 q \norm{\hat X}_2, \\
& (1-\delta-\zeta_2 q) \|M\|^2_F \leq \mathbf m^\top \mathbf{H} \mathbf m \leq \\ &(1+\delta+\zeta_2 q)  \|M\|^2_F, \ \ \forall M: \rk(M) \leq 2r .
\end{aligned}
\end{equation}
where $\mathbf m = \vecc(M)$. Note that since the second constraint is hard to deal with, so we solve the following problem that has the same optimal value (as proved in Lemma 14 of \cite{bi2020global}):
	\begin{equation}\label{eq:local_lmi_delta}
\begin{aligned}
\min_{\delta,\mathbf H} \quad & \delta \\
\st \quad & \norm{\hat{\mathbf X}^\top \mathbf H \mathbf e} \leq 2 \zeta_1 q \norm{\hat X}_2, \\
& (1-\delta-\zeta_2 q)I_{n^2} \preceq \mathbf H \preceq (1+\delta+ \zeta_2 q)I_{n^2}.
\end{aligned}
\end{equation}
If the optimal value of \eqref{eq:local_lmi_delta} is denoted as $\delta_f^*(\hat X)$, then we know that $\delta_f^*(\hat X) \leq \delta^*(\hat X) \leq \delta$ due to \eqref{eq:local_lmi_delta_2r} being a relaxation of \eqref{eq:intro_lmi}. By further lower-bounding $\delta_f^*(\hat X)$ with an expression in terms of $\|\hat X \hat X^\top - M^*\|_F$, we can obtain an upper bound on $\|\hat X \hat X^\top - M^*\|_F$. 

\begin{proof}[Proof of Lemma~\ref{lem:noisy_foc_nec}]
	Similar to the last section, we first define $\hat M = \hat X \hat X^\top$. Since $\hat X$ is a first-order critical point, it follows from \eqref{eq:foc} that $\nabla_X h(\hat X,w) = 0$. Thus,
	\begin{equation}
	\label{eq:foc_nec_lem_help_1}
		0 = \langle \nabla_X h(\hat X,w), U \rangle = \langle \nabla_M f(\hat M,w), \hat X U^\top+U \hat X^\top \rangle,
	\end{equation}
	for an arbitrary $U \in \mathbb{R}^{n \times r}$. Let $u = \vecc(U)$.
	
	Next, we define the function $g(\cdot): \mathbb{R}^{n \times n} \mapsto \mathbb{R}$:
	\begin{equation*}
		g(V) = \langle \nabla_M f(V,w), \hat X U^\top+U \hat X^\top \rangle,
	\end{equation*}
	for all $V \in \mathbb{R}^{n \times n}$. Then, $g(\hat M) = 0$ due to \eqref{eq:foc_nec_lem_help_1}.
	
	By the mean-value theorem (MTV), we have:
	\begin{equation*}
		\begin{aligned}
			g(\hat M) - g(M^*) &= \int_0^1 \langle \nabla g(tM^*+(1-t)\hat M), \hat M- M^* \rangle \mathrm{d} t \\
			&= \int_0^1 [ \nabla_M^2 f(tM^*+(1-t)\hat M)](\hat M- M^*,  \hat X U^\top+U \hat X^\top)  \mathrm{d} t \\
			&= \mathbf{e}^\top \mathbf{H} \mathbf{\hat X} u
		\end{aligned}
	\end{equation*}
	where $\mathbf{H} \in \RR^{n^2 \times n^2}$ is a symmetric matrix that is independent of $U$ and satisfies:
	\[
		\vecc(K)^\top \mathbf{H} \vecc(L) = \int_0^1 [ \nabla_M^2 f(tM^*+(1-t)\hat M)](K,L) \mathrm{d} t
	\]
	for all $K, L \in \RR^{n \times n}$. This means:
	\[
		\mathbf{e}^\top \mathbf{H} \mathbf{\hat X} u = g(
		\hat M) - g(M^*).
	\]
	Taking the absolute value of both sides and upper-bounding the right-hand side gives:
	\begin{equation*}
		\begin{aligned}
			|\mathbf{e}^\top \mathbf{H} \mathbf{\hat X} u| &= |g(
		\hat M) - g(M^*)| \leq |g(M^*)| \\
		&\leq \zeta_1 q \|\hat X U^\top+U \hat X^\top\|_F \\
		&\leq 2 \zeta_1 q \|\hat X U^\top\|_F \\
		&=2 \zeta_1 q \sqrt{\tr(\hat X \hat X^\top U U^\top)} \\
		&\leq 2 \zeta_1 q \|\hat X\|_2 \|u\|,
		\end{aligned}
	\end{equation*}
	where the second line follows from combining \eqref{eq:gt_nabla_zero} and \eqref{eq:noise_perturb_grad}, and the fourth line follows from the cyclic property of trace operators.
	
	Choosing $u = \hat{\mathbf X}^\top \mathbf{H} \mathbf{e}$ can simplify the above inequality to 
	\[
		\|\hat{\mathbf X}^\top \mathbf{H} \mathbf{e} \| \leq 2 \zeta_1 q \|\hat X\|_2.
	\]
	
	Furthermore, the $\delta$-RIP$_{2r,2r}$ property of the objective function means that:
	\[
		(1-\delta) \|M\|^2_F \leq [\nabla^2 f(\xi,0)](M,M) \leq (1+\delta) \|M\|^2_F
	\]
	for all $M$ with $\rk(M) \leq 2r$. Combining with the fact that
	\[
		| \vecc(M)^\top \mathbf{H} \vecc(M) - [\nabla^2 f(\xi,0)](M,M) | \leq \zeta_2 q \|M\|^2_F,
	\]
	gives \eqref{eq:foc_nec_lem_rip}.
\end{proof}

\begin{proof}[Proof of Theorem~\ref{thm:local}]

One can replace the decision variable $\delta$ in \eqref{eq:local_lmi_delta} with $\eta$ and introduce the following optimization problem:
\begin{equation}\label{eq:local_lmi_eta}
\begin{aligned}
\max_{\eta,\hat{\mathbf H}} \quad & \eta \\
\st \quad & \norm{\hat{\mathbf X}^\top\hat{\mathbf H}\mathbf e} \leq 2 \zeta_1 q \norm{\hat X}_2, \\
& \eta I_{n^2} \preceq \hat{\mathbf H} \preceq I_{n^2}.
\end{aligned}
\end{equation}
It is easy to realize that given any feasible solution $(\delta, \mathbf H)$ for \eqref{eq:local_lmi_delta}, the following pair of points will serve as a feasible solution to \eqref{eq:local_lmi_eta}:
\[
\eta = \frac{1-\delta - \zeta_2 q}{1+\delta + \zeta_2 q}, \qquad \hat{\mathbf H} = \frac{1}{1+\delta + \zeta_2 q} \mathbf H.
\]
By denoting the optimal value of \eqref{eq:local_lmi_eta} as $\eta_f^*(\hat X)$, it holds that
\begin{equation}\label{eqn:delta_eta}
\eta_f^*(\hat X) \geq \frac{1-\delta_f^*(\hat X) - \zeta_2 q}{1+\delta_f^*(\hat X) + \zeta_2 q}  \geq \frac{1-\delta - \zeta_2 q}{1+\delta + \zeta_2 q},
\end{equation}
for all local minimizers (it is important to recall $\delta_f^*(\hat X) \leq \delta^*(\hat X) \leq \delta$). 

As stated above, the key to proving \eqref{eq:local_range} is to upper-bounding $\eta_f^*(\hat X)$. Since \eqref{eq:local_lmi_eta} is a semidefinite programming problem, finding any feasible solution of its Lagrangian dual can provide an upper bound. The dual problem is given as follows:
\begin{equation}\label{eqn:eta_dual}
\begin{aligned}
\min_{U_1,U_2,G,\lambda,y} \quad & \tr(U_2)+4 \zeta_1^2 q^2 \norm{\hat X}^2_2 \lambda+\tr(G) \\
\st \quad & \tr(U_1)=1, \\
& (\hat{\mathbf X}y)\mathbf e^\top+\mathbf e(\hat{\mathbf X}y)^\top=U_1-U_2, \\
& \begin{bmatrix}
G & -y \\
-y^\top & \lambda
\end{bmatrix} \succeq 0, \\
& U_1 \succeq 0, \quad U_2 \succeq 0.
\end{aligned}
\end{equation}
As per \cite{ma2021sharp}, define
\[
M=(\hat{\mathbf X} y)\mathbf e^\top+\mathbf e(\hat{\mathbf X} y)^\top,
\]
and decompose $M$ as $M = [M]_+ - [M]_-$ with $[M]_+ \succeq 0$ and $[M]_- \succeq 0$. Then, we find a set of feasible solutions $(U_1^*,U_2^*,G^*,\lambda^*,y^*)$ to \eqref{eqn:eta_dual}, which are:
\begin{equation*}
	\begin{aligned}
		y&^* = \frac{y}{\tr([M]_+)}, \quad U_1^* =  \frac{[M]_+}{\tr([M]_+)}, \quad U_2^* = \frac{[M]_-}{\tr([M]_+)}, \\
		& \quad G^* = \frac{y^* (y^*)^\top}{ \lambda^*}, \quad \lambda^* = \frac{\| y^* \|}{2 \zeta_1 q \norm{ \hat X}_2}.
	\end{aligned}
\end{equation*}
It is easy to verify that the above solution is feasible and has the objective value 
\begin{equation}\label{eq:dualobj}
\frac{\tr([M]_-)+ 4 \zeta_1 q \norm{\hat X}_2 \| y\|}{\tr([M]_+)}.
\end{equation}

For any matrix $\hat X \in \RR^{n \times r}$ satisfying $\norm{\hat X\hat X^\top-M^*}_F \leq \tau \lambda_r(M^*)$, we have $\hat X \neq 0$. Moreover, it has been shown in the proof of Lemma 19 in \cite{bi2020global} that any $y \neq 0$ for which $\hat X^\top \mat(y)$ is symmetric satisfies the inequality
\begin{equation}
	\norm{\hat{\mathbf X}y}^2 \geq 2\lambda_{r^*}(\hat X\hat X^\top)\norm{y}^2.
	\label{eq:yineq1}
\end{equation}
where $r^*$ is the rank of $\hat X$. Furthermore, by the Wielandt--Hoffman theorem,
\begin{align*}
	&| \lambda_{r^*}(\hat X\hat X^\top) - \lambda_{r^*}(M^*)| \leq \|\hat X\hat X^\top-M^*\|_F \leq \tau \lambda_r(M^*), \\
	&| \lambda_1(\hat X\hat X^\top) - \lambda_1(M^*)| \leq \|\hat X\hat X^\top-M^*\|_F \leq \tau \lambda_r(M^*).
\end{align*}

Thus, using the above two inequalities and \eqref{eq:yineq1}, we have
\begin{equation}\label{eq:yineq3}
\begin{aligned}
	\frac{2 \norm{\hat X}_2 \|y\|}{\|\hat{\mathbf X} y\|} \leq \frac{2 \norm{\hat X}_2}{\sqrt{2 \lambda_{r^*}(\hat X\hat X^\top)}} \leq 
	\sqrt{\frac{2(\lambda_1(M^*)+\tau \lambda_r(M^*))}{(1-\tau)\lambda_r(M^*)} } \coloneqq C(\tau,M^*).
\end{aligned}
\end{equation}
The second inequality holds because
\begin{align*}
	\lambda_{r^*}(\hat X\hat X^\top) &= \lambda_{r^*}(M^*) - (\lambda_{r^*}(M^*) - \lambda_{r^*}(\hat X\hat X^\top)) \\
	&\geq \lambda_{r^*}(M^*) - |(\lambda_{r^*}(M^*) - \lambda_{r^*}(\hat X\hat X^\top)| \\
	&\geq \lambda_r(M^*) - \tau \lambda_r(M^*) = (1-\tau) \lambda_r(M^*)
\end{align*}
Next, according to Lemma 14 of \cite{zhang2019sharp}, one can write
\begin{align*}
	&\tr([M]_+) = \|\hat{\mathbf X} y\| \|\mathbf e\| (1+ \cos \theta), \\
	&\tr([M]_-) = \|\hat{\mathbf X} y\| \|\mathbf e\| (1- \cos \theta).
\end{align*}
where $\theta$ is the angle between $\hat{\mathbf X} y$ and $\mathbf e$.
Substituting the above two equations and \eqref{eq:yineq3} into the dual objective value \eqref{eq:dualobj}, one can obtain
\[
\eta_f^*(\hat X) \leq \frac{1- \cos \theta+2 \zeta_1 q C(\tau,M^*)/\norm{\mathbf e}}{1+\cos \theta},
\]
which together with \eqref{eqn:delta_eta} implies that
\begin{equation}
	\norm{\mathbf e} \leq \frac{(1+\delta+\zeta_2 q) \zeta_1 q C(\tau,M^*)}{\cos\theta- \zeta_2 q -\delta}.
	\label{eq:local_ineq_int}
\end{equation}
Now, we seek to lower-bound $\cos(\theta)$. This amounts to taking the upper bound of $\sin^2(\theta)$. This requires us to choose a particular value of $y$. We choose the same $y$ that is described in Lemma 12 of \cite{zhang2020many}, since it makes $\hat X^\top \mat(y)$ symmetric, thereby satisfying \eqref{eq:yineq1}. From the proof of Lemma 13 of \cite{zhang2020many}, we know:
\begin{align*}
	\sin^2(\theta) = \frac{\|Z^\top(I - \hat X \hat X^\dagger) Z\|^2_F}{\|\hat X \hat X^\top - ZZ^\top \|^2_F},
\end{align*}

Since the expression of $\sin^2(\theta)$ is invariant to re-scaling, we may re-scale both $\hat X$ and $Z$ until $\|ZZ^\top \|^2_F=1$. Also, since the expression is rotationally invariant, we can partition $\hat X$ and $Z$ as follows:
\begin{equation*}
	\hat X = \begin{bmatrix}
		X_1 \\ 0
	\end{bmatrix} \hspace{2em} Z = \begin{bmatrix}
		Z_1 \\
		Z_2
	\end{bmatrix}
\end{equation*}
where $X_1, Z_1 \in \mathbb{R}^{r \times r}, Z_2 \in \mathbb{R}^{(n-r) \times r}$. We compute the QR decomposition $QR = [X,Z]$ and redefine $X \coloneqq Q^\top X, Z \coloneqq Q^\top Z$. Then, we follow the technique in Lemma 13 to arrive at:
\begin{equation*}
\begin{aligned}
	\frac{\|Z^\top(I - \hat X \hat X^\dagger) Z\|^2_F}{\|\hat X \hat X^\top - ZZ^\top \|^2_F} = \frac{\|Z_2 (Z_2)^\top\|_F^2}{\|Z_1 (Z_1)^\top - X_1 X_1^\top\|_F^2 + 2\|Z_1 (Z_2)^\top \|^2_F + \|Z_2 (Z_2)^\top\|_F^2}.
\end{aligned}
\end{equation*}
Additionally,
\begin{equation} \label{eq:z1_lower}
	\begin{aligned}
		\sigma_{\min}^2(Z_1) =&  \lambda_{\min}((Z_1)^\top(Z_1))\\
		 \geq& \lambda_{\min}((Z_1)^\top(Z_1)+(Z_2)^\top(Z_2)) - \lambda_{\max}((Z_2)^\top Z_2) \\
	=& \sigma^2_{r}(Z) - \|Z_2(Z_2)^\top\|_2 \\
	\geq &  \sigma^2_{r}(Z) - \tau \lambda_{r}(M^*) = (1-\tau) \lambda_{r}(M^*).
	\end{aligned}
\end{equation}
The last line of \eqref{eq:z1_lower} is due to
\begin{equation} \label{eq:z2_upper}
	\begin{aligned}
		\tau^2 \lambda^2_{r}(M^*) \geq& \|\hat X \hat X^\top - ZZ^\top \|^2_F \\
		= &\|Z_1 (Z_1)^\top - X_1 X_1^\top\|_F^2 + 2\|Z_1 (Z_2)^\top \|^2_F + \|Z_2 (Z_2)^\top\|_F^2 \\
	 \geq & \|Z_2 (Z_2)^\top\|^2_F,
\end{aligned}
\end{equation}
and that $\|Z_2 (Z_2)^\top\|_F \geq  \|Z_2 (Z_2)^\top\|_2$.

Subsequently, 
\begin{align*}
	\sin^2(\theta) &\leq  \frac{\|Z_2(Z_2)^\top\|^2_F}{2\|Z_1 (Z_2)^\top \|^2_F + \|Z_2 (Z_2)^\top\|_F^2 } \\
	&\leq \frac{\|Z_2(Z_2)^\top\|_F \|Z_2\|^2_F}{ 2\sigma_{\min}^2(Z_1) \|Z_2\|^2_F+ \|Z_2 (Z_2)^\top\|_F \|Z_2\|^2_F } \\
	&\leq \frac{\|Z_2(Z_2)^\top\|_F}{2(1-\tau) \lambda_{r}(M^*) + \|Z_2 (Z_2)^\top\|_F } \\
	&\leq \frac{\tau \lambda_{r}(M^*)}{2(1-\tau) \lambda_{r}(M^*) + \tau \lambda_{r}(M^*) } \\
	&\leq \frac{\tau}{ (2-\tau)} \leq \tau,
\end{align*}
where the first inequality follows from the fact that $\|Z_1 (Z_1)^\top - X_1 X_1^\top\|_F^2 \geq 0$, the third inequality follows from \eqref{eq:z1_lower}, and the fourth inequality follows from \eqref{eq:z2_upper} and the fact that the function $\frac{x}{c+x}$ is increasing with $x$ when both $c$ and $x$ are positive. 

The above bound is automatically non-vacuous, since $\sin^2(\theta) \leq \tau < 1$. Therefore,
\begin{equation*}
	\cos \theta \geq \sqrt{1-\tau},
\end{equation*}
leading to \eqref{eq:local_range} after substitution into \eqref{eq:local_ineq_int}.

\end{proof}

\subsection{Proof of Section~\ref{sec:convergence_local}}
\label{sec:appendix_linear_cov}
First and foremost, we restate this lemma from \cite{tu2016low,zhu2018global}:
\begin{lemma}
	\label{lem:dis_conversion}
	For any matrix $X \in \mathbb{R}^{n \times r}$, given a positive semidefinite matrix $M \in \mathbb{R}^{n \times n}$ of rank $r$, we have:
	\begin{equation}
		\label{eq:dist_change}
		\|XX^\top - M\|^2_F \geq 2 (\sqrt{2} -1) \sigma_r(M) (\dist(X,M))^2.
	\end{equation}
\end{lemma}
Also, given Assumption 5, we have
\begin{equation}
	\label{eq:go_nabla_zero}
	\nabla_M f(M^w,w) = 0
\end{equation}

First, we establish that the PL inequality holds in a neighborhood of the global minimizer.
\begin{lemma}
	\label{lem:pl_ineq}
	Consider the global minimizer $M^w$ of \eqref{eq:unfactored_noisy_problem}. There exists a constant $\mu > 0$ such that the PL inequality:
	\begin{equation}
		\label{eq:pl_ineq}
		\frac{1}{2} \| \nabla_X h(X,w)\|^2_F \geq \mu (h(X,w)-f(\projr(M^w), w)),
	\end{equation}
	holds for all $X \in \mathbb{R}^{n \times r}$ satisfying:
	\begin{align}
		\label{eq:pl_region}
			\dist(X,M^w) < \max\{\sqrt{2(\sqrt 2 -1)} \sqrt{1-(\delta+\zeta_2 q)^2} (\sigma_r(M^w))^{1/2}-D_r,0 \}
	\end{align}
	and
	\[
		D_r \leq \dist(X,\projr(M^w)),
	\]
	for $q < (1-\delta)/\zeta_2$. 
\end{lemma}
	
\begin{proof}[Proof of Lemma~\ref{lem:pl_ineq}]
	We prove the Lemma when $C_w \sqrt{1-(\delta+\zeta_2 q)^2}-D_r > 0$, since otherwise it is trivial. Denote $M \coloneqq XX^\top$. First, we fix a constant $\tilde C$ such that:
\begin{equation}\label{eq:tildeC_ineq}
	\dist(X,M^w) \leq \tilde C < C_w \sqrt{1-(\delta+\zeta_2 q)^2} -D_r.
\end{equation}
Then, we define $q_1$ and $q_2$ as follows:
\begin{equation}
	\begin{aligned}
		q_1 = \sqrt{1- \frac{\tilde C^2}{2(\sqrt 2 -1)\sigma_r(M^w)}}, q_2 = \frac{\sqrt 2 \mu'}{\sigma_r(M^w)^{1/2}-\tilde C}.
	\end{aligned}
\end{equation}
Now, both $q_1$ and $q_2$ are nonnegative resulting from the assumption above. Furthermore, we know that $\delta+\zeta_2 q < \sqrt{1- \frac{\tilde C^2}{2(\sqrt 2-1)\sigma_r(M^w)}}$ from \eqref{eq:tildeC_ineq}, then
\begin{equation}
	\label{eq:pl_ineq_q1q2}
	\frac{1-\delta-\zeta_2 q}{1+\delta + \zeta_2 q} > \frac{1-q_1+q_2}{1+q_1},
\end{equation}
for some small enough $\mu'$. Define $\mu = (\mu')^2/(1+\delta+\zeta_2 q + 2\rho)$. First, we make the assumption that:
\begin{equation}
	\label{eq:inverse_pl}
	\frac{1}{2} \| \nabla_X h(X,w)\|^2_F < \mu (h(X,w)-f(\projr(M^w), w)).
\end{equation}
From this assumption, we have:
\begin{align*}
	\mu (h(X,w)-f(\projr(M^w), w)) &\leq \mu \left( \langle \nabla_M f(\projr(M^w),w), M- \projr(M^w) \rangle + \frac{1+\delta+\zeta_2 q}{2}\|M-\projr(M^w)\|^2_F \right) \\
	&\leq \mu \left( \rho \|M^w -\projr(M^w)\|_F \|M-\projr(M^w)\|_F + \frac{1+\delta+\zeta_2 q}{2}\|M-\projr(M^w)\|^2_F \right) \\
	&\leq \mu \left( \rho \|M-\projr(M^w)\|^2_F + \frac{1+\delta+\zeta_2 q}{2}\|M-\projr(M^w)\|^2_F \right).
\end{align*}
due to Taylor's theorem and \eqref{eq:noise_perturb_hess}. So then \eqref{eq:inverse_pl} leads to:
\[
	\frac{1}{2} \| \nabla h(X,w)\|^2_F < \mu(\frac{ (1+\delta+\zeta_2q)}{2} + \rho)\|M - \projr(M^w)\|_F^2.
\]
Therefore, 
\[
	\| \nabla h(X,w)\|_F \leq \mu' \|M-\projr(M^w)\|_F.
\]
Then consider the following optimization problem:
\begin{equation}
	\label{eq:pl_lmi}
	\begin{aligned}
		\min_{\delta,\mathbf{H} \in \mathbb{S}^{n^2}} \quad & \delta \\
\st \quad & \norm{\hat{\mathbf X}^\top \mathbf H \mathbf e} \leq \mu' \|\mathbf e\|, \\
& \text{$\mathbf H$ satisfies the $(\delta+\zeta_2 q)$-$\RIP_{2r}$ property}.
	\end{aligned}
\end{equation}
where $\mathbf e = \vecc(XX^\top - \projr(M^w))$. If we denote the optimal value of \eqref{eq:pl_lmi} as $\delta^*_f(X,\mu')$, then $\delta^*_f(X,\mu') \leq \delta$ because the constraints of \eqref{eq:pl_lmi} are necessary conditions for \eqref{eq:inverse_pl}, according to Lemma 12 of \cite{bi2021local}. Therefore,
\[
	\frac{1-\delta - \zeta_2 q}{1+\delta+\zeta_2 q} \leq \frac{1-\delta^*_f(X,\mu') - \zeta_2 q}{1+\delta^*_f(X,\mu')+\zeta_2 q}.
\]
Moreover, by the same logic of \eqref{eqn:delta_eta}, we know that $\eta_f^*(X,\mu') \geq \frac{1-\delta_f^*(X,\mu') - \zeta_2 q}{1+\delta_f^*(X,\mu') + \zeta_2 q}$, where $\eta_f^*(X,\mu')$ is the optimal value of the optimization problem:
\begin{equation}
\begin{aligned}
\max_{\eta,\hat{\mathbf H}} \quad & \eta \\
\st \quad & \norm{\hat{\mathbf X}^\top\hat{\mathbf H}\mathbf e} \leq \mu' \|\mathbf e\|, \\
& \eta I_{n^2} \preceq \hat{\mathbf H} \preceq I_{n^2}.
\end{aligned}
\end{equation}
Lemma 14 of \cite{bi2021local} gives:
\[
	\eta_f^*(X,\mu') \leq \frac{1-q_1+q_2}{1+q_1},
\]
therefore making a contradiction to \eqref{eq:pl_ineq_q1q2}, subsequently proving \eqref{eq:pl_ineq}. 
\end{proof}

\begin{proof}[Proof of Theorem \ref{thm:linear_conv}]
If we certify that:
	\begin{equation}
		\label{eq:pl_ineq_matrix}
		\frac{\|XX^\top - M^w\|_F}{C_w} <  C_w \sqrt{1-(\delta+\zeta_2 q)^2}  -D_r
	\end{equation}
	for any given $X \in \mathbb{R}^{n \times r}$, then a direct substitution can certify that \eqref{eq:pl_region} holds for $X$, since by Lemma \ref{lem:dis_conversion},
	\[
		\dist(X,M) \leq \frac{\|XX^\top - M^w\|_F}{C_w}.
	\]
	Therefore, the certification of  \eqref{eq:pl_ineq_matrix} means that the PL inequality \eqref{eq:pl_ineq} holds for this given $X$. Given that \eqref{eq:pl_init} is satisfied, then if this inequality holds:
	\begin{equation}
		\label{eq:pl_ineq_matrix_simp}
		\|XX^\top - M^w\|_F \leq \sqrt{\frac{1+\delta+\zeta_2 q}{1 -\delta - \zeta_2 q}} \|X_0 X_0^\top - M^w \|_F,
	\end{equation}
	 \eqref{eq:pl_ineq_matrix} will also hold, because:
	 \[
	 	\sqrt{\frac{1+\delta+\zeta_2 q}{1 -\delta - \zeta_2 q}} \|X_0 X_0^\top - M^w \|_F \leq C^2_w\sqrt{1-(\delta+\zeta_2 q)^2} -C_w D_r.
	 \]
	 Thus, for the remainder of the proof, we aim to certify that starting from $X_0$, if we apply the gradient descent algorithm, \eqref{eq:pl_ineq_matrix_simp} will be satisfied every step along this trajectory.
	
	In order to do so, we use Taylor's expansion and \eqref{eq:go_nabla_zero} to obtain
	\[
		f(M,w) - f(M^w,w) =\frac{[\nabla^2 f(N,w)](M-M^w,M-M^w)}{2},
	\]
	where $N$ is some convex combination of $M$ and $M^w$, and $M \in \mathbb{R}^{n \times n}$ is any matrix of rank at most $r$. In light of the RIP property of the function and \eqref{eq:noise_perturb_hess}, one can write: 
	\[
		\frac{1-\delta-\zeta_2 q}{2} \|M - M^w\|^2_F \leq f(M,w) - f(M^w,w) \leq \frac{1+\delta +\zeta_2 q}{2} \|M - M^w\|^2_F.
	\]
	This means that if $M_1,M_2 \in \mathbb{R}^{n \times n}$ are two matrices of rank at most $r$ with $f(M_1,w) \leq f(M_2,w)$, then:
	\begin{equation}
		\|M_1 - M^w \|_F \leq \sqrt{\frac{1+\delta+\zeta_2 q}{1 -\delta - \zeta_2 q}} \|M_2 - M^w \|_F,
	\end{equation}
	because $f(M_1,w) - f(M^w,w) \leq f(M_2,w) - f(M^w,w)$.
	
	Thus, one can conclude that $f(X_t X_t^\top,w) \leq f(X_0 X_0^\top,w) \ \forall t$, where $X_t$ denotes the $t^{\text{th}}$ step of the gradient descent algorithm starting from $X_0$. Hence, \eqref{eq:pl_ineq_matrix_simp} follows for all $X_t$.
	
	Conveniently, Lemma 11 in \cite{bi2021local} shows that $f(X_t X_t^\top,0) \leq f(X_{t-1} X_{t-1}^\top,0)$ for all $t \geq 0$. However, this result can be extended to:
	\[
		f(X_t X_t^\top,w) \leq f(X_{t-1} X_{t-1}^\top,w),
	\]
	by making
	\begin{align*}
		1/\eta \geq 12 \rho r^{(1/2)}\left(\sqrt{\frac{1+\delta+\zeta_2 q}{1 -\delta - \zeta_2 q}} \|X_0 X_0^\top - M^w \|_F + \|M^w\|_F \right),
	\end{align*}
	since $\nabla f(\cdot,w)$ is now a $\rho$-Lipschitz continuous function. Given \eqref{eq:pl_init}, a sufficient condition to the above inequality is that:
	\[
		\eta \leq \left(12 \rho r^{(1/2)}\left(2(\sqrt 2 -1)\sqrt{(1-(\delta+\zeta_2 q)^2}  + \|M^w\|_F \right)\right)^{-1}
	\]
	
	This finally means that the PL inequality \eqref{eq:pl_ineq} is established for the entire trajectory starting from $X_0$. Now, applying Theorem 1 in \cite{karimi2016linear} gives:
	\[
		h(X_t,w)-f(\projr(M^w),w) \leq (1-\mu \eta)^t(h(X_0,w)-f(\projr(M^w),w)),
	\]
	which implies a linear convergence as desired.
	
\end{proof}

\subsection{Proof Sketches in Section \ref{sec:strict_saddle} }
\label{sec:appendix_strict_saddle}
The proof of Theorem \ref{thm:strict_saddle} is highly similar to that of Theorem 7 in \cite{zhang2021general}, albeit with a number of differences. In this section, we will only highlight the differences, since everything else follows in the same manner.

First and foremost, we replace $\delta$ with $\delta+\zeta_2 q$ in all of the proofs since in our noisy formulation, the problem is $(\delta+\zeta_2 q)$-RIP$_{2r,2r}$ instead.

Then, we introduce the following Lemma in lieu of Lemma 6 in \cite{zhang2021general} since $\nabla_M f(M^*,w) \neq 0$ in the noisy formulation:
\begin{lemma}
	\label{lem:new_lemma6}
	Given a constant $\epsilon > 0$, an arbitrary $X \in \mathbb{R}^{n \times r}$, and the ground truth solution $M^* \in \mathbb{R}^{n \times n}$ of \eqref{eq:unfactored_noisy_problem}, if 
	\begin{equation}\label{eq:M_fro_bound}
		\|XX^\top\|^2_F \geq \max \left\{ \frac{2(1+\delta+\zeta_2 q )}{1-\delta - (\zeta_2+\zeta_D) q} \|M^*\|^2_F, (\frac{2\lambda \sqrt r}{1-\delta - (\zeta_2+\zeta_D) q})^{4/3} \right\},
	\end{equation}
	then
	\[
		\|\nabla_X h(X,w)\|_F \geq \lambda,
	\]
	where $\zeta_D = \zeta_1/D$ and $D$ is a constant such that 
	\begin{equation}\label{eq:D_bound}
		D^2 \leq (\frac{2\lambda \sqrt r}{1-\delta - (\zeta_2+\zeta_D) q})^{4/3}.
	\end{equation}
\end{lemma}
Note that such $D$ exists since we first require that $1-\delta - (\zeta_2+\zeta_D) q \geq 0$, meaning that $\frac{q\zeta_1}{1-\delta-q\zeta_2 }\leq D$. Moreover, a sufficient condition to \eqref{eq:D_bound} is that $D \leq (2\lambda \sqrt r)^{2/3} $, which can be simultaneously satisfied when $\lambda$ is chosen properly. The introduction of the lower bound $D$ will not affect the remainder of the proof of Theorem \ref{thm:strict_saddle}, since in the later steps, we only require the existence of a constant $C$ such that $\|XX^\top\|_F \leq C^2$ when $\|\nabla_X h(X,w)\|_F \leq \lambda$. Therefore, Lemma \ref{lem:new_lemma6} perfectly fits this role.
\begin{proof}[Proof of Lemma \ref{lem:new_lemma6}]
		Denote $M \coloneqq XX^\top$. Using the RIP property and \eqref{eq:noise_perturb_grad}, we have:
		\begin{align*}
			\langle \nabla_M f(M), M \rangle &= \int_0^1 [\nabla^2 f(M^*+s(M-M^*),w)][M-M^*,M] \de s+ \langle \nabla_M f(M^*,w), M \rangle \\
			&\geq (1-\delta-\zeta_2 q) \|M\|_F^2 - (1+\delta+\zeta_2 q) \|M^*\|_F \|M\|_F- \zeta_1 q \|M\|_F \\
			&= (1-\delta-\zeta_2 q) \|M\|_F^2 - (1+\delta+\zeta_2 q) \|M^*\|_F \|M\|_F- \zeta_D q D \|M\|_F \\
			& \geq (1-\delta-(\zeta_2+\zeta_D) q) \|M\|_F^2 - (1+\delta+\zeta_2 q)\|M^*\|_F \|M\|_F \\
			& \geq \frac{1-\delta-(\zeta_2+\zeta_D) q}{2} \|M\|_F^2,
		\end{align*}
		where the second last inequality results from \eqref{eq:D_bound}, which implies that $D \leq \|M\|_F$; and the last inequality follows from \eqref{eq:M_fro_bound}. Then combining the fact that $\|X\|_F \leq \sqrt r \|M\|^{1/2}_F$, and $\|\nabla_X h(X,w)\|_F \geq \frac{\langle \nabla h(X,w),X \rangle }{\|X\|_F}$ yields the desired fact that
		\begin{equation}
			\begin{aligned}
				\|\nabla_X h(X,w)\|_F \geq \frac{\langle \nabla h(X,w),X \rangle }{\|X\|_F} &= \frac{\langle \nabla_M f(M), M \rangle}{\|X\|_F} \\
				&\geq \frac{(1-\delta-(\zeta_2+\zeta_D) q) \|M\|_F^2}{2\sqrt r \|M\|^{1/2}_F} \\
				&= \frac{1-\delta-(\zeta_2+\zeta_D) q}{2\sqrt r} \|M\|^{3/2}_F \\
				&\geq \lambda.
			\end{aligned}
		\end{equation}
\end{proof} 

Then, utilizing Lemma \ref{lem:new_lemma6}, we can prove Lemma 7 in \cite{zhang2021general} in the same fashion to obtain
\begin{align*}
	\langle \nabla_M f(M,w), M^* - M \rangle &\leq -(1-\delta - \zeta_2 q) \|M-M^*\|^2_F - \langle \nabla_M f(M^*,w), M-M^* \rangle \\
	&\leq -(1-\delta - \zeta_2 q) \|M-M^*\|^2_F + \zeta_1 q \|M-M^*\|_F \\
	& \leq -(1-\delta - \zeta_2 q) \|M-M^*\|^2_F + \zeta_\alpha q (\sqrt{2(\sqrt 2-1)} (\sigma_r(M^*))^{1/2} \alpha) \|M-M^*\|_F \\
	& \leq -(1-\delta - (\zeta_2-\zeta_\alpha) q) \|M-M^*\|^2_F
\end{align*}
for any $M \in \mathbb{R}^{n \times n}$ that satisfies the requirements in Lemma 7 of \cite{zhang2021general}. This is because $\|M-M^*\|_F \geq (\sqrt{2(\sqrt 2-1)} (\sigma_r(M^*))^{1/2} \alpha)$ by the assumption of $\alpha$ and Lemma \ref{lem:dis_conversion}.

The above change will only affect the constant $c$ in Lemma 7, and the new $c$ will become
\[
	c = (\sqrt r \|M^*\|_F)^{-1} (\sqrt 2 -1)(1-\delta-(\zeta_2-\zeta_\alpha)q)  \sigma_r(M^*).
\]
Since the exact value of $c$ is irrelevant and we only need to prove its existence, the rest of the proof follows from the existing procedure. Note that $c > 0$ is guaranteed by the assumption of noise in Theorem \eqref{thm:strict_saddle}. Therefore, Lemma 7 still holds in the noisy case.

Then, we proceed to show that Lemma 8 in \cite{zhang2021general} can also be proved similarly, except for one key difference, which is
\[
	K \coloneqq (1-3\delta -(3 \zeta_2 +2 \zeta_\alpha) q)(\sqrt 2 -1) \sigma_r(M^*) \alpha^2.
\]
To verify this statement, we leverage the inequality
\[
	-\phi(\bar M) \geq f(M,w) - f(M^*,w) - (\delta+\zeta_2 q)\|M-M^*\|_F^2,
\]
and furthermore we now have that
\begin{align*}
	f(M,w) - f(M^*,w) &\geq \langle \nabla_M f(M^*,w), M-M^* \rangle + \frac{1-\delta-\zeta_2 q}{2} \|M-M^*\|_F^2 \\
	&\geq \frac{1-\delta-\zeta_2 q}{2} \|M-M^*\|_F^2 - \zeta_1 q \|M-M^*\|_F^2 \\
	&\geq \frac{1-\delta-\zeta_2 q}{2} \|M-M^*\|_F^2 - \zeta_\alpha q (\sqrt{2(\sqrt 2-1)} (\sigma_r(M^*))^{1/2} \alpha) \|M-M^*\|_F^2 \\
	&\geq \frac{1-\delta-(\zeta_2+2\zeta_\alpha) q}{2} \|M-M^*\|_F^2
\end{align*}
for the same reason elaborated above. Combining the above two inequalities leads to
\[
	-\phi(\bar M) \geq \frac{1-3\delta -(3\zeta_2 + 2\zeta_\alpha)q}{2} \|M-M^*\|_F^2 \geq K.
\]
As assumed in Theorem \ref{thm:strict_saddle}, since $q < \frac{1/3-\delta}{\zeta_2+2\zeta_\alpha/3}$, we know that $K >0$. This is the only required property of $K$ to facilitate the remainder of the proof of Lemma 8 of \cite{zhang2021general}. Therefore, Lemma 8 still holds for the noisy case.

Finally, we choose $C = (\frac{2(1+\delta+\zeta_2 \epsilon )}{1-\delta - (\zeta_2+\zeta_D) \epsilon} \|M^*\|^2_F)^{1/4}$ and invoke Lemmas 6-8 to complete the proof of Theorem \ref{thm:strict_saddle}. Note the $\epsilon$ here is the same $\epsilon$ appeared in the statement of Theorem \ref{thm:strict_saddle}.

\newpage
\section{Additional Numerical Illustration}\label{sec:figures}

\begin{figure}[!h]
    \centering
    \begin{subfigure}{8cm}
    	    \includegraphics[width=\linewidth]{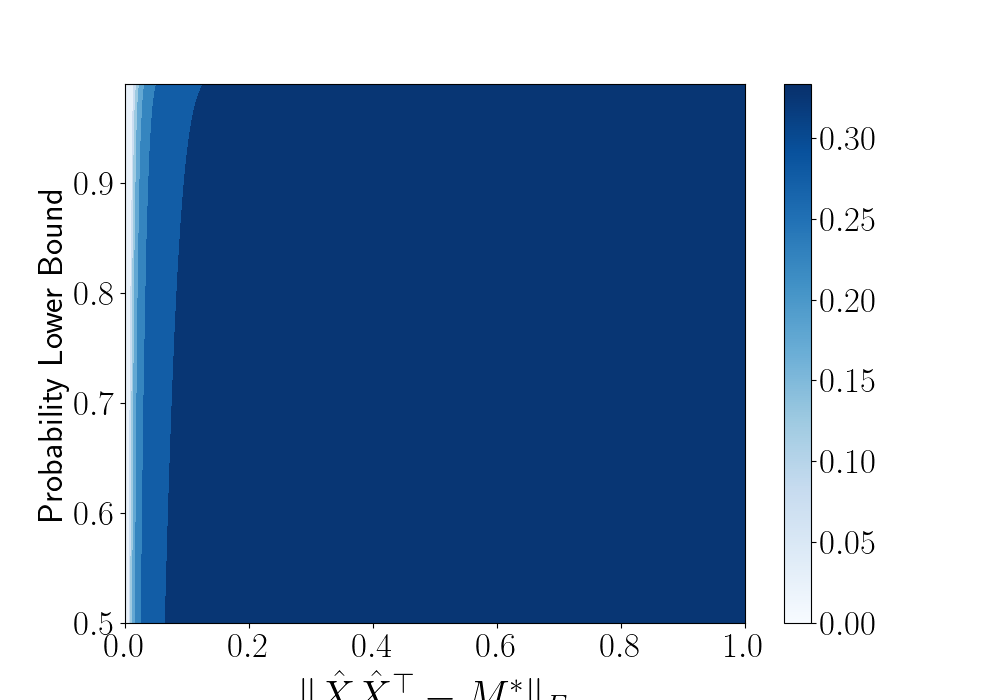}
    \caption{$\delta$ bound in Theorem~\ref{thm:global_local_min} when $\zeta_1 = 0.001$.}
    \end{subfigure} \hspace{2em}
    \begin{subfigure}{8cm}
    	    \includegraphics[width=\linewidth]{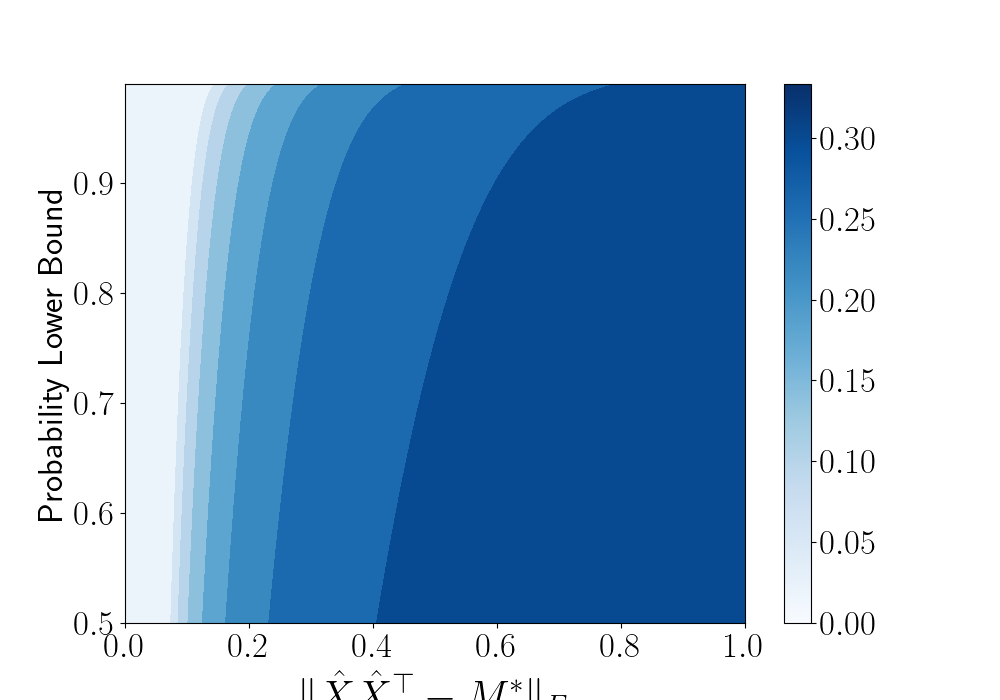}
    \caption{$\delta$ bound in Theorem~\ref{thm:global_local_min} when $\zeta_1 = 0.01$.}
    \end{subfigure}\vspace{2em}\\
    \begin{subfigure}{8cm}
    	    \includegraphics[width=\linewidth]{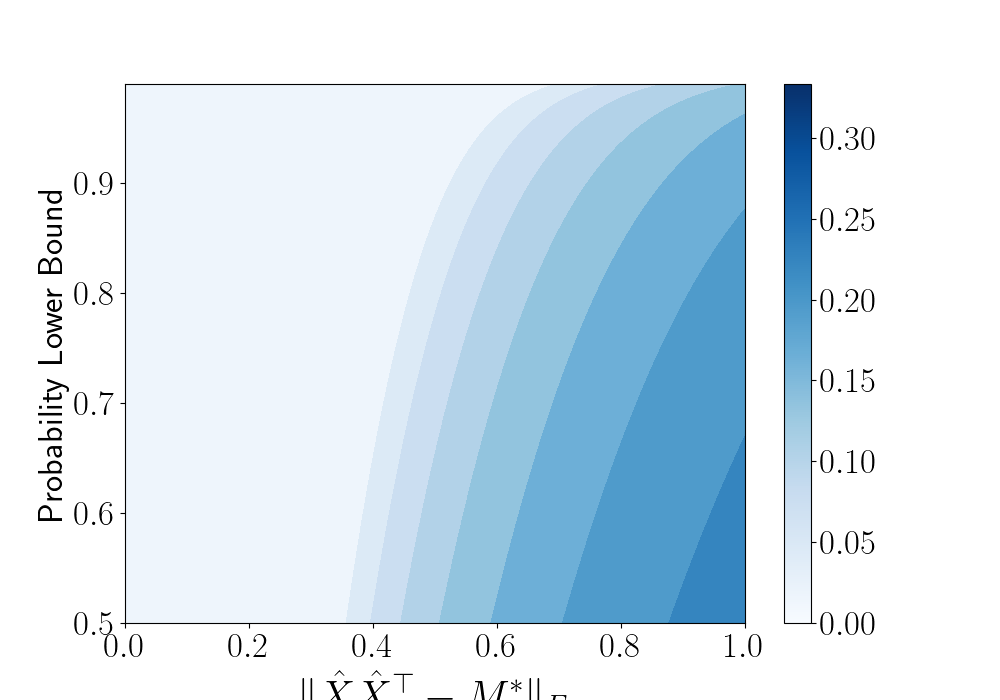}
    \caption{$\delta$ bound in Theorem~\ref{thm:global_local_min} when $\zeta_1 = 0.05$.}
    \end{subfigure} \hspace{2em}
    \begin{subfigure}{8cm}
    	    \includegraphics[width=\linewidth]{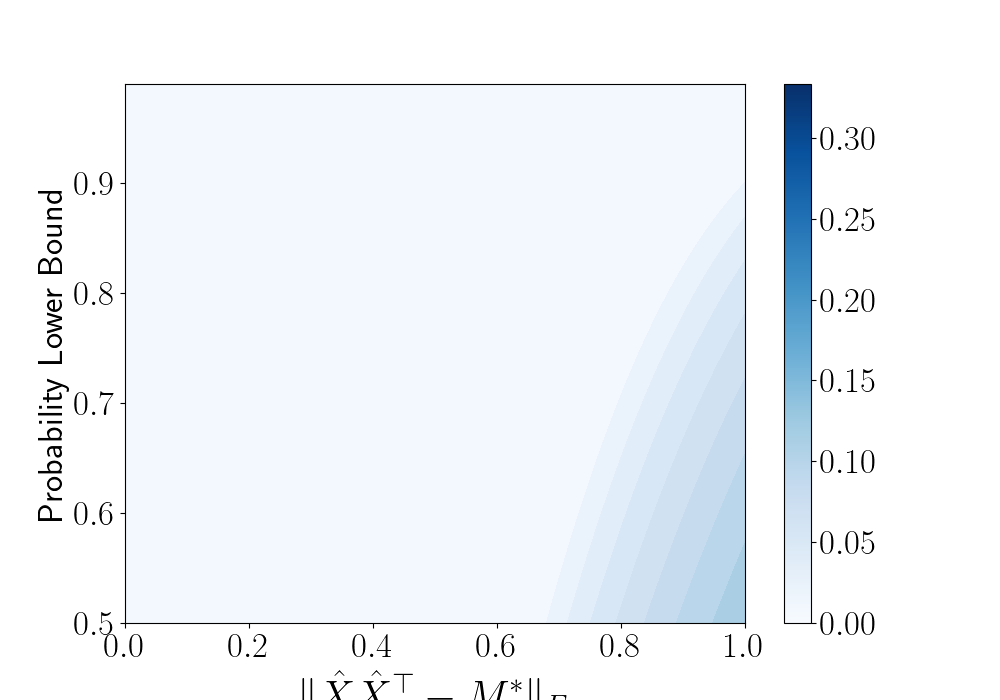}
    \caption{$\delta$ bound in Theorem~\ref{thm:global_local_min} when $\zeta_1 = 0.1$.}
    \end{subfigure}\vspace{2em}\\
    \caption{Comparison of the maximum RIP constants $\delta$ allowed by Theorem~\ref{thm:global_local_min} to guarantee a given bound on the distance $\|\hat X\hat X^\top - M^*\|_F$ for an arbitrary local minimizer $\hat X$ satisfying \eqref{eq:tau_range} with a given probability. In this plot $\zeta_1=0, \sigma = 0.05$ as per the main text.}
\end{figure}

\begin{figure}[!h]
    \centering
    \begin{subfigure}{8cm}
    	    \includegraphics[width=\linewidth]{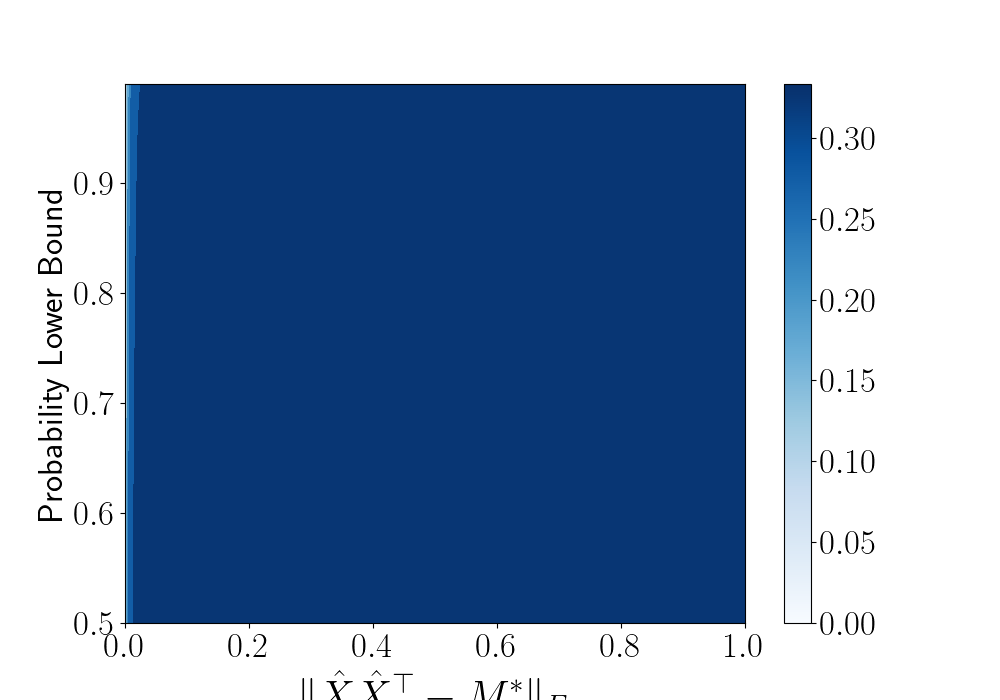}
    \caption{$\delta$ bound in Theorem~\ref{thm:global_local_min} when $\sigma = 0.0001$.}
    \end{subfigure} \hspace{2em}
    \begin{subfigure}{8cm}
    	    \includegraphics[width=\linewidth]{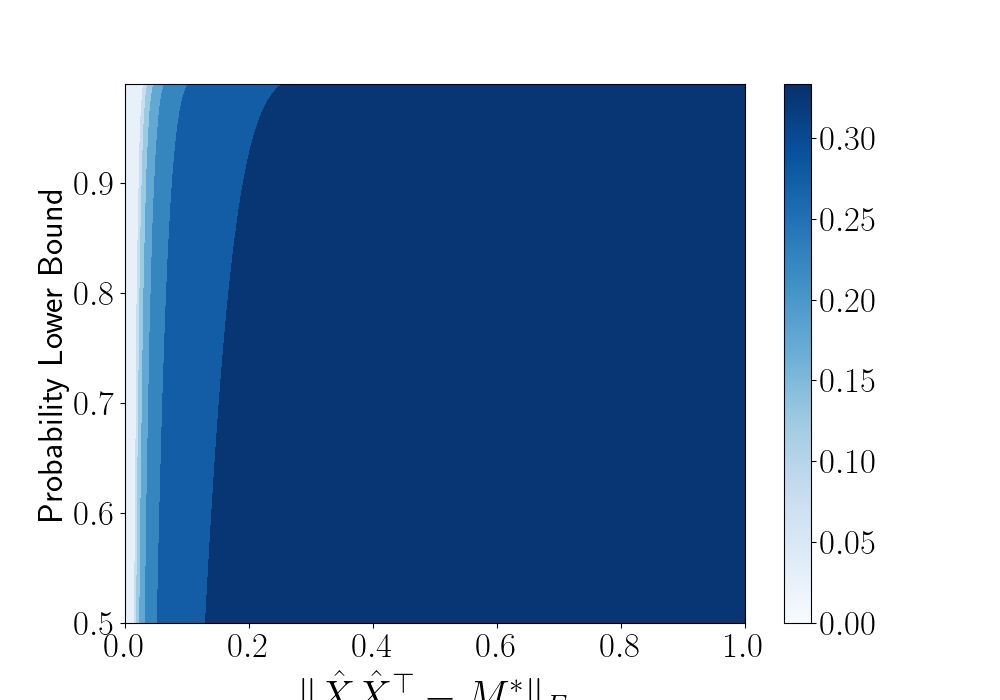}
    \caption{$\delta$ bound in Theorem~\ref{thm:global_local_min} when $\sigma = 0.01$.}
    \end{subfigure}\vspace{2em}\\
    \begin{subfigure}{8cm}
    	    \includegraphics[width=\linewidth]{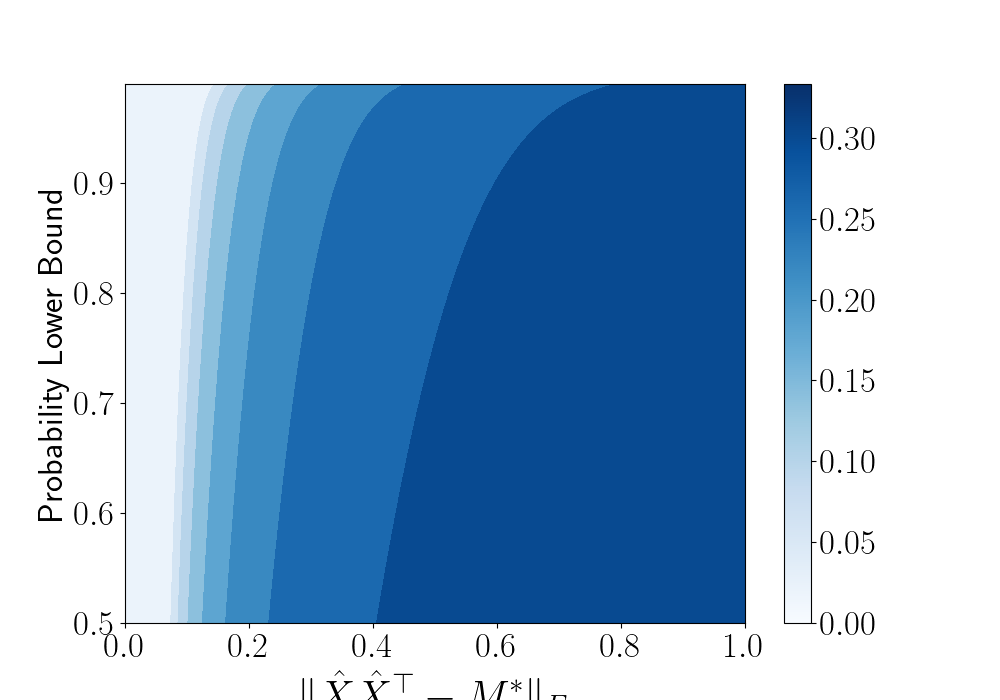}
    \caption{$\delta$ bound in Theorem~\ref{thm:global_local_min} when $\sigma = 0.05$.}
    \end{subfigure} \hspace{2em}
    \begin{subfigure}{8cm}
    	    \includegraphics[width=\linewidth]{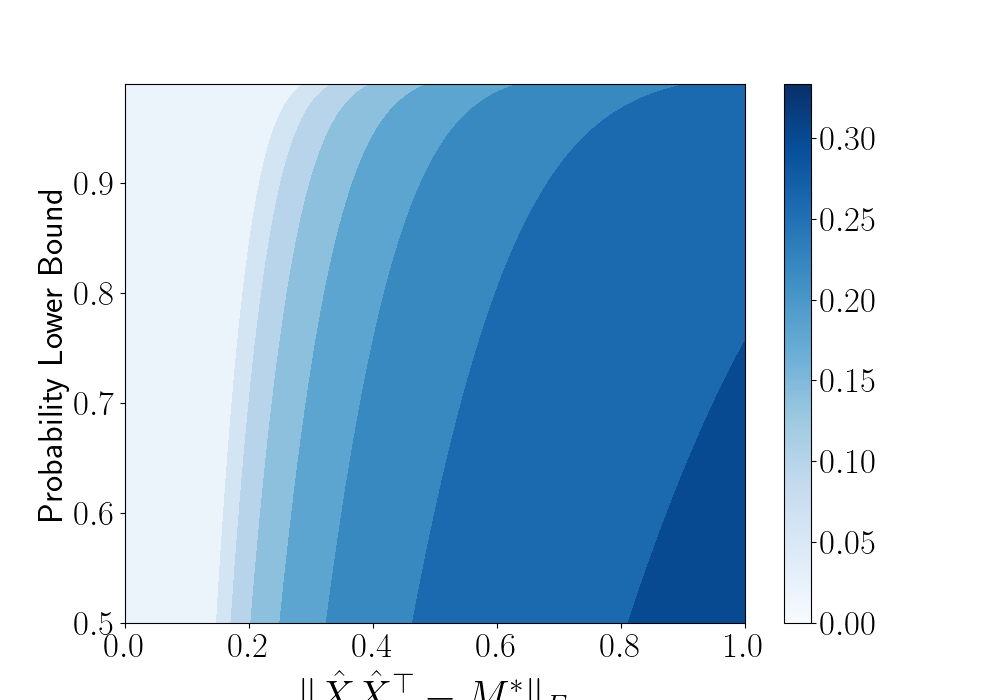}
    \caption{$\delta$ bound in Theorem~\ref{thm:global_local_min} when $\sigma = 0.1$.}
    \end{subfigure}\vspace{2em}\\
    \caption{Comparison of the maximum RIP constants $\delta$ allowed by Theorem~\ref{thm:global_local_min} to guarantee a given bound on the distance $\|\hat X\hat X^\top - M^*\|_F$ for an arbitrary local minimizer $\hat X$ satisfying \eqref{eq:tau_range} with a given probability. In this plot $\zeta_1=0.01, \zeta_2 = 0$.}
\end{figure}

\begin{figure}[!h]
    \centering
    \begin{subfigure}{8cm}
    	    \includegraphics[width=\linewidth]{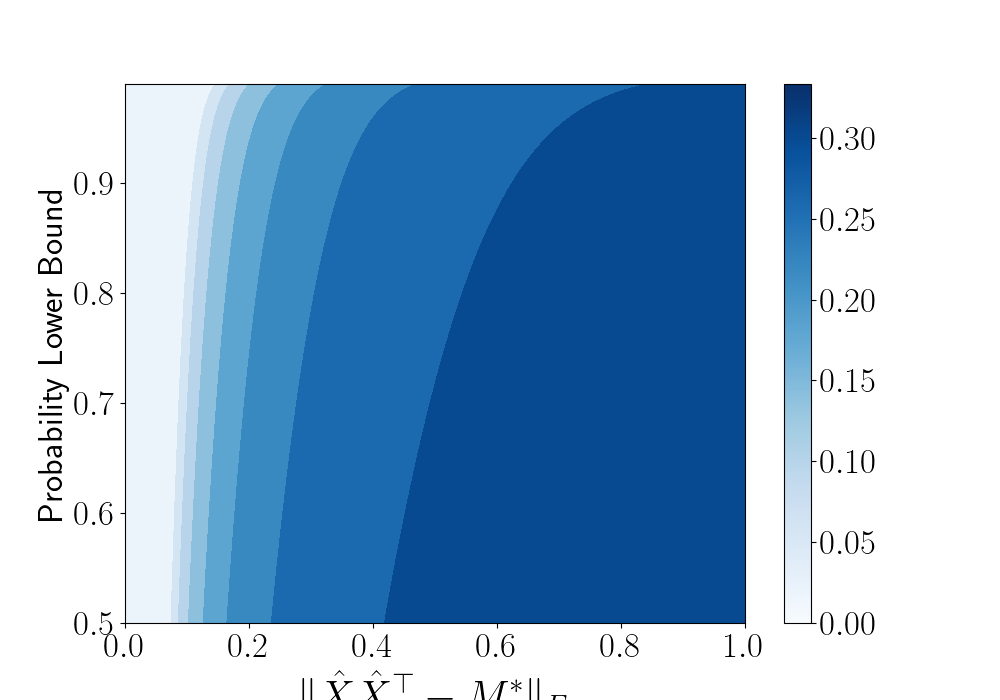}
    \caption{$\delta$ bound in Theorem~\ref{thm:global_local_min} when $\zeta_2 = 0.0005$.}
    \end{subfigure} \hspace{2em}
    \begin{subfigure}{8cm}
    	    \includegraphics[width=\linewidth]{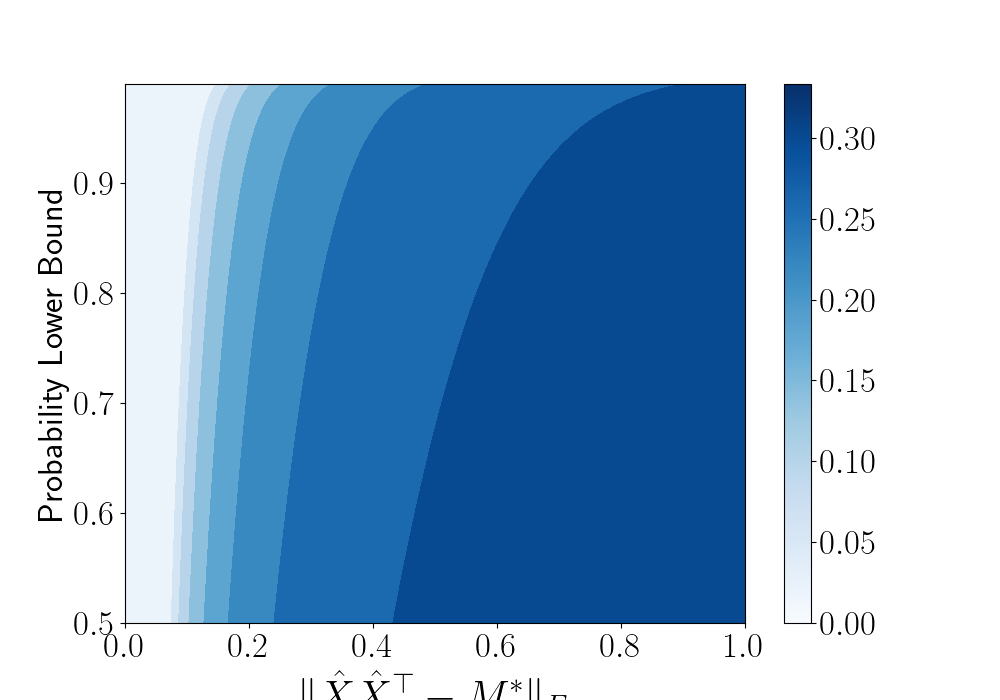}
    \caption{$\delta$ bound in Theorem~\ref{thm:global_local_min} when $\zeta_2 = 0.001$.}
    \end{subfigure}\vspace{2em}\\
    \begin{subfigure}{8cm}
    	    \includegraphics[width=\linewidth]{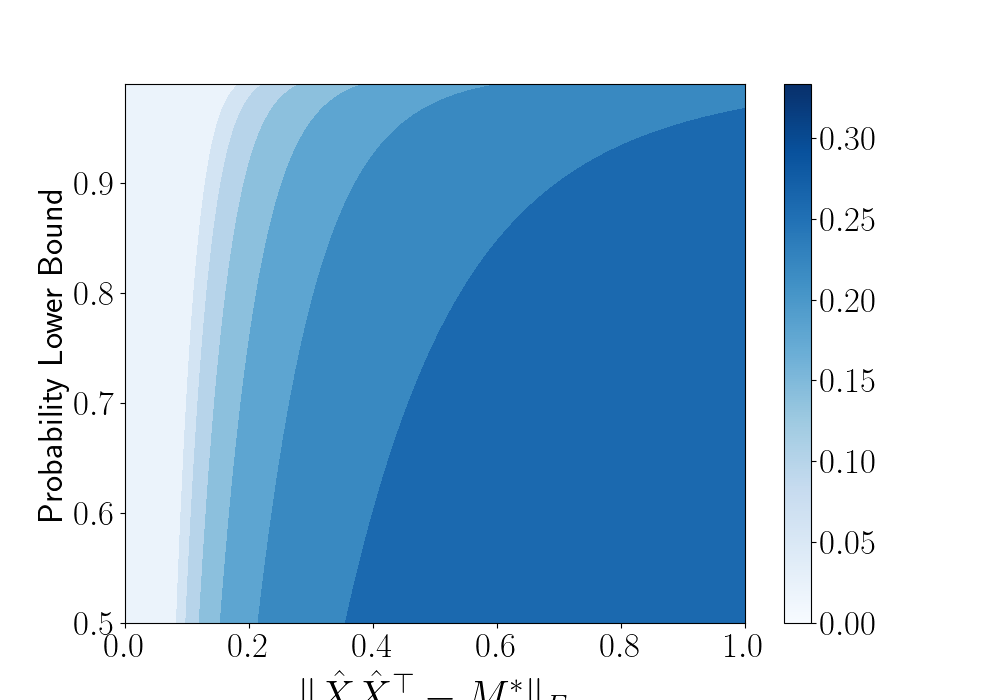}
    \caption{$\delta$ bound in Theorem~\ref{thm:global_local_min} when $\zeta_2 = 0.01$.}
    \end{subfigure} \hspace{2em}
    \begin{subfigure}{8cm}
    	    \includegraphics[width=\linewidth]{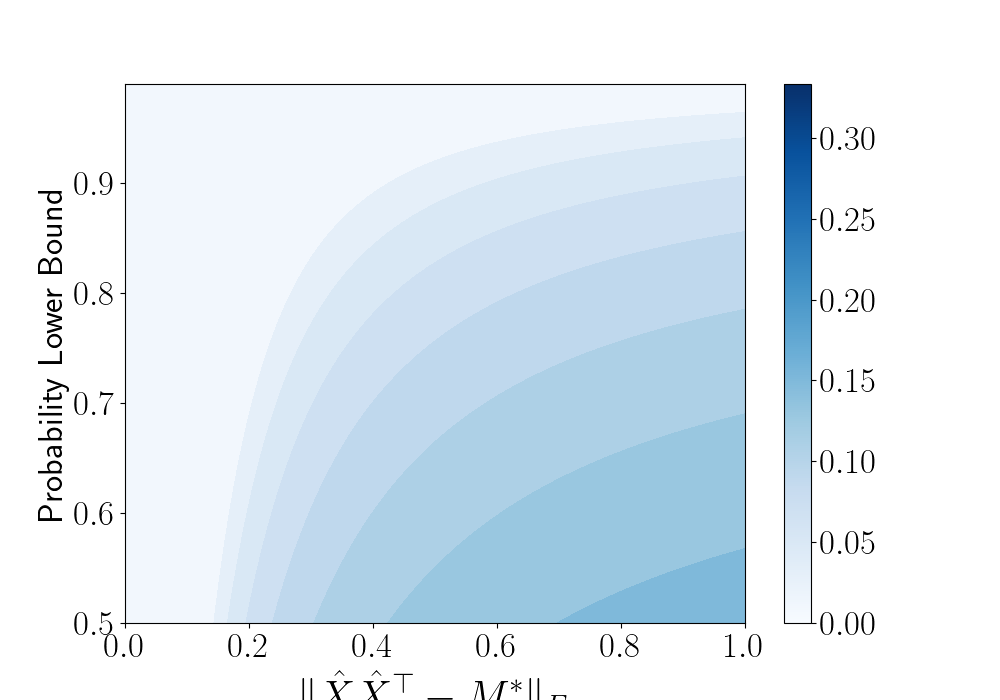}
    \caption{$\delta$ bound in Theorem~\ref{thm:global_local_min} when $\zeta_2 = 0.05$.}
    \end{subfigure}\vspace{2em}\\
    \caption{Comparison of the maximum RIP constants $\delta$ allowed by Theorem~\ref{thm:global_local_min} to guarantee a given bound on the distance $\|\hat X\hat X^\top - M^*\|_F$ for an arbitrary local minimizer $\hat X$ satisfying \eqref{eq:tau_range} with a given probability. In this plot $\zeta_1=0.01, \sigma = 0.05$.}
\end{figure}
\end{document}